\documentclass[12pt]{amsart}
\usepackage{amsthm,amsmath,a4,amssymb,verbatim,url,enumerate}
\usepackage[latin1]{inputenc}

\newcommand{\ignorecontentsline}[3]{}
\let\trueaddcontentsline=\addcontentsline

\newtheorem{thm}{Theorem}[section]
\newtheorem{cor}[thm]{Corollary}
\newtheorem{lem}[thm]{Lemma}
\newtheorem{prop}[thm]{Proposition}
\theoremstyle{definition}
\newtheorem{defn}[thm]{Definition}
\theoremstyle{remark}
\newtheorem{rem}[thm]{Remark}
\newtheorem{warn}[thm]{Warning}
\newtheorem{ex}[thm]{Example}
\newtheorem{que}[thm]{Question}

\newcommand{\eps}{\varepsilon}
\newcommand{\g}{\mathfrak{g}}
\newcommand{\N}{\mathbf{N}}
\newcommand{\Z}{\mathbf{Z}}
\newcommand{\R}{\mathbf{R}}
\newcommand{\Q}{\mathbf{Q}}
\newcommand{\C}{\mathbf{C}}

\begin{document}

\address{CNRS and Univ Lyon, Univ Claude Bernard Lyon 1, Institut Camille Jordan, 43 blvd. du 11 novembre 1918, F-69622 Villeurbanne}
\email{cornulier@math.univ-lyon1.fr}



\title[On sublinear bilipschitz equivalence of groups]{On sublinear bilipschitz equivalence of groups}
\author{Yves Cornulier}%
\date{March 20, 2018}
\thanks{Partially supported by ANR Project ANR-14-CE25-0004 GAMME}

\begin{abstract}
We discuss the notion of sublinearly bilipschitz equivalences (SBE), which generalize quasi-isometries, allowing some additional terms that behave sublinearly with respect to the distance from the origin. Such maps were originally motivated by the fact they induce bilipschitz homeomorphisms between asymptotic cones. We prove here that for hyperbolic groups, they also induce H\"older homeomorphisms between the boundaries. This yields many basic examples of hyperbolic groups that are pairwise non-SBE. Besides, we check that subexponential growth is an SBE-invariant.

The central part of the paper addresses nilpotent groups. While classification up to sublinearly bilipschitz equivalence is known in this case as a consequence of Pansu's theorems, its quantitative version is not. We introduce a computable algebraic invariant $e=e_G<1$ for every such group $G$, and check that $G$ is $O(r^e)$-bilipschitz equivalent to its associated Carnot group. Here $r\mapsto r^e$ is a quantitive sublinear bound.

Finally, we define the notion of large-scale contractable and large-scale homothetic metric spaces. We check that these notions imply polynomial growth under general hypotheses, and formulate conjectures about groups with these properties.



\end{abstract}


\maketitle
\medskip

\renewcommand{\thesubsection}{{\thesection.\Alph{subsection}}} 

\let\addcontentsline=\ignorecontentsline
\section{Introduction}

\subsection{Sublinearly Lipschitz maps}

We consider here some functions between metric spaces, generalizing large-scale Lipschitz maps and quasi-isometries. 

Let $v$ be a real-valued function on $\R_+$. For our purposes, typical examples of $v$ are $v(r)=r^\alpha$ for some $\alpha\in [0,1]$, or $v(r)=\log(r)$. See \S\ref{s_sbm} for precise (mild) hypotheses. We assume here $v\ge 1$, up to replacing $v$ with $\max(v,1)$ if necessary. 

We say that a map $f:X\to Y$ between metric spaces is $O(v)$-Lipschitz if it satisfies
\[d(f(x),f(x'))\le Cd(x,x')+C'v(|x|+|x'|),\qquad\forall x,x'\in X,\]
for some constants $C,C'>0$. Here, $|x|$ denotes the distance from $x$ to some base-point of $X$ (fixed once and for all, but whose choice does not matter). We also say that $f$ is $o(v)$-Lipschitz if it is $O(v')$-Lipschitz for some $v'=o(v)$; in particular, for $v(r)=r$, we call $o(r)$-Lipschitz maps sublinearly Lipschitz maps; they were introduced in \cite{CoI} under the name ``cone-Lipschitz maps".

For instance, $O(1)$-Lipschitz maps are known as large-scale Lipschitz maps, and occur naturally in the large-scale category, whose isomorphisms are quasi-isometries; see for instance \cite[Chap.\ 3]{CH}. 

We are especially interested in $O(v)$-Lipschitz maps when $v(r)=o(r)$ (that is, sublinearly Lipschitz maps). Indeed, for instance $O(r)$-Lipschitz maps are just maps with a radial control $|f(x)|\le C|x|+C'$ and are thus of limited interest.

There is a natural equivalence relation in the set of $O(v)$-Lipschitz maps $X\to Y$, called $O(v)$-closeness. Namely, $f,f'$ are $O(v)$-close if $d(f(x),f'(x))\le C''v(|x|)$ for all $x\in X$, and some $C''>0$. Similarly, $f,f'$ are $o(v)$-close if they are $O(v')$-close for some $v'=o(v)$. 

The sublinearly Lipschitz category, consists in metric spaces as objects, with the set of arrows $X\to Y$ being the set of $o(r)$-Lipschitz maps up to $o(r)$-equivalence. This category was introduced in \cite{CoI} with the following motivation: taking asymptotic cones (with respect to a given scaling sequence and ultrafilter) yields a functor from this category to the category of metric spaces with Lipschitz maps. Moreover, this is, in a precise sense \cite[Prop.\ 2.9]{CoI}, the largest setting for which such functors can be defined. Isomorphisms in this category are called sublinearly Lipschitz equivalences, or SBE maps (``cone-bilipschitz equivalences" in \cite{CoI}). A simple verification shows that a map $f:X\to Y$ is an SBE if and only there exists a (locally bounded above) function $v=o(r)$ such that it satisfies the following three conditions:
\begin{itemize}
\item $f$ is $O(v)$-Lipschitz: there exist constants $c,C>0$ 
such that for all $x,x'\in X$ one has
\[ d(f(x),f(x'))\le cd(x,x')+Cv(|x|+|x'|);\]
\item $f$ is $O(v)$-expansive: there exist constants $c',C'>0$ such that for all $x,x'\in X$ one has
\[c'd(x,x')-C'v(|x|+|x'|)\le d(f(x),f(x'));\]
\item $f$ is $O(v)$-surjective: there exist a constant $C''>0$ such that, for all $y\in Y$, one has
\[d(y,f(X))\le C''v(|y|).\]
\end{itemize}

Note that multiplying $v$ by a scalar (depending on $f$) allows to get rid of the constants $C,C',C''$.

\begin{ex}
1) If $\pi$ is any sublinear function $\R\to\R$, then $x\mapsto x+\pi(x)$ is an SBE $\R\to\R$ (being $o(r)$-close to the identity map).

2) Let $X$ be the set of square integers in $\mathbf{R}_+$ and let $f$ be the map $x\mapsto \lfloor\sqrt{x}\rfloor^2$ from $\mathbf{R}_+$ to $X$. Then $f$ is an SBE, but is not $o(r)$-close to any large-scale Lipschitz map: indeed by a simple geodesic argument, every large-scale Lipschitz map $\mathbf{R}_+\to X$ is bounded and in particular $X$ and $\mathbf{R}_+$ are not quasi-isometric.
\end{ex}

There are more elaborate examples, which are important in geometric group theory. For simply connected nilpotent Lie groups, the classification up to quasi-isometry is conjectural: it is expected that quasi-isometric implies isomorphic. Nevertheless, Pansu obtained the first non-trivial quasi-isometric rigidity results by proving important theorems about their asymptotic cones in the eighties, which can be restated in terms of SBEs. 
Namely every simply connected nilpotent Lie group is SBE to a Carnot group (called its associated Carnot-graded group) \cite{Pan}, and any two SBE Carnot groups are isomorphic \cite{Pan2}. Later, Shalom \cite{sh04} proved, with unrelated methods, that the Betti numbers of the Lie algebra of a simply connected nilpotent Lie group are quasi-isometry invariants, while taking associated Carnot-graded group does not always preserves the Betti numbers. Thus there exist simply connected nilpotent Lie groups (with lattices) that are SBE but not quasi-isometric. SBEs allow to restate Pansu's theorems with no reference to asymptotic cones (the asymptotic cone theorems, also related to Goodman's earlier work \cite{Goo}, being corollaries), but also yields other interpretations. We come back to this topic in \S\ref{nigr}.

One can naturally generalize the sublinearly Lipschitz category and define, in a similar way, the $O(v)$-Lipschitz category and the $o(v)$-Lipschitz category. In particular, the $O(1)$-Lipschitz category is known as the large-scale Lipschitz category. Thus these interpolate between the large-scale category and the sublinearly Lipschitz category. There are obvious inclusion functors from the $O(v)$-category to the $O(v')$-category whenever $v=O(v')$ (this is usually not faithful, because of the equivalence relation). Isomorphisms in the $O(v)$-Lipschitz or $o(v)$-Lipschitz category are called $O(v)$-SBE or $o(v)$-SBE (assuming that $v=o(r)$). 

For instance, it is established in \cite{CoI} that every connected Lie group is $O(\log r)$-SBE to a Lie group of the form $G=N\rtimes E$ with both $N,E$ simply connected nilpotent Lie groups, $N$ being exponentially distorted in $G$ and $E$ acting in a diagonalizable way on the Lie algebra of $N$.

\subsection{SBEs and growth}

In \S\ref{growth_ame}, we prove the following:

\begin{thm}
1) Subexponential growth is SBE-invariant among connected graphs of bounded valency, and in particular among compactly generated locally compact groups.

2) (Folklore up to the formulation) Polynomial growth is SBE-invariant among compactly generated locally compact groups.
\end{thm}

In spite of their similarity, these results are very different in nature. The first, about subexponential growth, is very general and the proof is direct. The second is not true in the graph setting (it is very easy to check that $\Z$ is SBE to a connected graph of valency $\le 3$ and growth greater than any polynomial. However, for groups, polynomial growth can be characterized by a large-scale doubling property and this gives the result. Note that this is equivalent to the result that groups polynomial growth are precisely those with all their asymptotic cones proper; a fact which is known to experts.

It would be natural to wonder about SBE-invariance of growth and amenability among groups; I have included some open questions in \S\ref{growth_ame}.

\subsection{SBE and ends}\label{i_sbends}

Given a geodesic metric space $X$, there is a natural way to define its end space $E(X)$, using discrete paths going to infinity (in the metric sense); see \S\ref{ss_ends}. When $X$ is proper (or quasi-isometric to a proper space), $E(X)$ is a compact space.

We say that $f:X\to Y$ is radially expansive if there exist constants $C,C'>0$ such that $|f(x)|\ge C|x|-C'$ for all $x$ (this does not depend on the choice of base-points; if $X=\emptyset$ we agree that $f$ is radially expansive). 

\begin{thm}\label{ends_func}
Every radially expansive sublinearly lipschitz map $f:X\to Y$ between connected graphs naturally induces a H\"older continuous map $f_*:E(X)\to E(Y)$. This assignment is functorial from the subcategory of the sublinearly Lipschitz category, where objects are connected graphs and maps are radially expansive maps up to sublinear closeness, to the category of metric spaces and H\"older continuous maps.
\end{thm}

In the subcategory described in the theorem, it is clear that the isomorphisms are the SBEs. Therefore we obtain:

\begin{cor}\label{c_SBE_end}
Any SBE $X\to Y$ between connected graphs induces a homeomorphism $E(X)\to E(Y)$, which is H\"older. 
\end{cor}

\begin{ex}\label{free_surface}
A non-abelian free group $F$ and a surface group $S$ (of a surface with negative Euler characteristic) are not SBE, although they have isometric asymptotic cones. Indeed, $E(F)$ is a Cantor set while $E(S)$ is a singleton.
\end{ex}

\subsection{SBE of hyperbolic groups}\label{intro_hyp}

If $X$ is a geodesic metric space, then $X$ is Gromov-hyperbolic if and only if all its asymptotic cones are real trees, for every choice of sequence of base-points; this is due to Gromov \cite[\S 2.A]{Gro}, with a more complete proof given by Drutu \cite[Prop.\ 3.4.1]{Dru}. Under sufficient homogeneity assumptions, this can be reduced to asymptotic cones with a fixed base-point, and thus this implies (see Proposition \ref{hyp_sbe_inv}):

\begin{prop}Gromov-hyperbolicity is an SBE-invariant among compactly generated locally compact groups (and in particular among finitely generated groups).\end{prop}

In the setting of connected graphs (which is more general, since we treat these groups using their Cayley graphs with respect to a compact generating subset), easy counterexamples show that this does not hold without homogeneity assumptions (Example \ref{sbe_noninv}).

In contrast to nilpotent groups, hyperbolic groups essentially all share the same asymptotic cone and thus the latter is of no further help in the SBE classification. Example \ref{free_surface}, based on Corollary \ref{c_SBE_end} shows that, however, the SBE-classification of non-elementary hyperbolic groups is non-trivial. Still, we cannot expect much more from Corollary \ref{c_SBE_end} since the space of ends of a non-elementary group is either a point or a Cantor set. Here we refine Corollary \ref{c_SBE_end} to show that the whole boundary is an SBE-invariant.

\begin{thm}\label{i_hyp_bou}
Let $f:X\to Y$ be a radially expansive (see \S\ref{i_sbends}) sublinearly Lipschitz map between proper geodesic Gromov-hyperbolic spaces. 
Then, $f$ induces a H\"older map $(\partial X,d_\mu)\to (\partial Y,d_\mu)$ between the visual boundaries. In particular, if $f$ is an SBE, then it induces a H\"older homeomorphism between visual boundaries.
\end{thm}

Here $\mu>0$ is a parameter appearing in the classical definition of the metric on the visual boundary; it has to be sufficiently small, in terms of the hyperbolicity constants of $X$ and $Y$.

\begin{cor}\label{corhy}
If hyperbolic groups have non-homeomorphic boundaries, then they are not SBE. 
\end{cor}

\begin{ex}
We retrieve the fact (Example \ref{free_surface}) that free group and a (closed) surface group are not SBE, as their boundaries are homeomorphic to a Cantor set, and a circle respectively. Corollary \ref{corhy} also distinguishes between 1-ended groups: a surface group is not SBE to a cocompact lattice in $\mathrm{PSL}_2(\C)$, and also between $\infty$-ended groups: for instance a free group and a free product of two surface groups are not SBE. The latter fact is also a consequence of Corollary \ref{rigi} below, proved in \S\ref{s_hyp}.
\end{ex}

\begin{cor}\label{rigi}
1) A compactly generated locally compact group is SBE to a free group if and only it admits a geometric (=proper cocompact continuous) action on a tree of finite valency. In particular, in the discrete case this characterizes virtually free groups.

2) A compactly generated locally compact group is SBE to a surface group if and only it admits a geometric action on the hyperbolic plane.
\end{cor}
In both cases, the corollary follows from the fact that these particular groups are characterized by the topology of their boundary. Note that this is very specific. If the Cannon conjecture holds (asserting that every discrete hyperbolic group whose boundary is homeomorphic to the 2-sphere acts geometrically on the hyperbolic 3-space), then this could be extended to locally compact groups SBE to the hyperbolic 3-space. Nevertheless, this simplified argument would fail in higher dimension; for instance $\mathrm{SU}(2,1)$ and its cocompact lattices admit no geometric action on the hyperbolic 4-space although they admit a homeomorphic copy of the 3-sphere as boundary.

\begin{rem}\label{termy}
SBE spaces were called cone-bilipschitz equivalent spaces in \cite{CoI}. But a free group and a surface group have bilipschitz-equivalent asymptotic cones but are not SBE (Example \ref{free_surface}): this is why we avoid this misleading terminology: bilipschitz maps between the asymptotic cones are not induced by maps between the groups. Note that in \cite{CoI}, it was established that given any two metric spaces, $X,Y$ any map $X\to Y$ inducing a bilipschitz homeomorphism between all asymptotic cones is an SBE.
\end{rem}

Corollary \ref{corhy} motivates the following questions, where by hyperbolic group we mean a locally compact group that is Gromov-hyperbolic for the word metric with respect to a compact generating subset. The two main popular examples are the case of finitely generated groups, and the case of Lie groups with a left-invariant negatively curved Riemannian metric.

\begin{que}\label{exhoho}
Does there exist any two hyperbolic groups with homeomorphic, but not H\"older-homeomorphic boundaries?
\end{que}

Let us point out that the boundary of a simply connected negatively curved $d$-dimensional manifold with a cocompact isometry group (discrete or not) is H\"older-homeomorphic to a round $(d-1)$-dimensional sphere. In particular, connected Lie groups will not provide examples for Question \ref{exhoho} (although the quasi-isometric classification therein is a rich problem, see \cite{Csu}).

\begin{que}\label{hynosb}
Does there exist any two hyperbolic groups that are not SBE but whose boundaries are homeomorphic? H\"older-homeomorphic? [Update: see Remark \ref{pallier}.]
\end{que}

\begin{ex}[SBE Lie groups that are not QI]\label{exr2uni}
For $a\in\{0,1\}$, let $G_a$ be the semidirect product $\R^2\rtimes\R$, where the action is given  by $t\cdot (x,y)=e^t(x+aty,y)$. Then $G_a$ admits a left-invariant negatively curved Riemannian metric. For $a=0$, it can be chosen to be of constant curvature $-1$; in particular $G_0$ is quasi-isometric to the real hyperbolic 3-space. It was proved by Xie \cite{Xie} that $G_1$ is not quasi-isometric to $G_0$. Nevertheless, $G_1$ and $G_0$ are SBE (and actually $O(\log r)$-SBE), see \cite[Theorem 4.4]{CoI}.
\end{ex}

The group $G_1$ has no lattice and is actually not quasi-isometric to any finitely generated group (this follows from the same result of Xie \cite{Xie}, see the discussion in \cite[\S 6]{Csu}). Therefore we can ask the question in the discrete case:

\begin{que}\label{hynonQI}
Does there exist any two discrete hyperbolic groups that are SBE but not quasi-isometric?
\end{que}

The last question concerns some particular important examples.

\begin{que}\label{semi}~
\begin{enumerate}
\item Consider the semidirect product $G_\alpha=\R^2\rtimes\R$, where $t\in\R$ acts through the diagonal matrix $(e^t,e^{\alpha t})$. Are the $G_\alpha$, for $\alpha\ge 1$, pairwise non-SBE?

\item\label{Drutu} (Suggested by C. Dru\c tu.) Are $\mathbb{H}^4_\R$ and $\mathbb{H}^2_\C$ non-SBE? The question can be extended to all pairs of non-homothetic negatively curved symmetric spaces of the same dimension. [Update: see Remark \ref{pallier}.]
\end{enumerate}
\end{que}

\begin{rem}\label{pallier}
Gabriel Pallier \cite{Pal} has announced a positive answer to Dru\c tu's question (in its extended version), thereby answering positively Question \ref{hynosb}.
\end{rem}

\subsection{On finite presentability}

Among finitely generated groups, finite presentability is a quasi-isometry invariant. More generally, among compactly generated locally compact groups, compact presentability is a quasi-isometry invariant. Indeed, it can be characterized as large-scale simple connectedness, a purely metric notion (see for instance \cite[Chap.~8]{CH}). 
Unsuccessful attempts suggest that it has no reason to be an SBE invariant in general. However there exists a stronger version of finite presentability that is SBE-invariant, namely the linear isodiametric filling property (LID). We define it precisely in \S\ref{SBEfp}; it roughly means that loops of length $n$ can be filled within a ball of radius $O(n)$. 

Many reasonable groups are known to be LID: combable and 3-manifold groups \cite{Ge}, central-by-hyperbolic groups \cite{Con}, virtually connected Lie groups (and hence their cocompact lattices, e.g., polycyclic groups) \cite[Theorem 3.B.1 and Remark 3.B.4]{CT}. It is also very likely that the geometric proofs that non-cocompact lattices in connected Lie groups are finitely presented with at most exponential Dehn function, also imply that they are LID. However, Gersten \cite{Ge} proved that there exist 1-relator groups that are not LID (actually, with iterated exponential as isodiametric function), for instance the Baumslag-Gersten group $\langle x,y\mid x^{x^y}=x^2\rangle$.

\begin{thm}
Among compactly generated locally compact groups, being LID is SBE-invariant.
\end{thm}

This implies that groups SBE to most known finitely/compactly presented groups are also finitely/compactly presented. For instance, we deduce:

\begin{cor}
Every finitely generated group SBE to a connected Lie group (or a cocompact lattice therein) is finitely presented. More generally, every compactly generated locally compact SBE to a connected Lie group is compactly presented.
\end{cor}

The corollary was already clear in case the group has all its asymptotic cones simply connected, because this property passes to SBEs and implies compact presentability. However, it is new, for instance, in the case of the group SOL, which has the same asymptotic cones as the (infinitely presented) lamplighter group \cite[\S 9]{CoJT}.

The theorem holds in some further generality for metric spaces, under a QI-homogeneity assumption, but not in general, as it is easy to show that $\Z$ is SBE to a graph of valency $\le 3$ that is not LID (and even not large-scale simply connected, see Example \ref{sbe_noninv}).

The questions remain open for groups:

\begin{que}
1) Does there exist a finitely presented group $G$ that is SBE to an infinitely presented (finitely generated) group?

2) What if $G$ is the Baumslag-Gertsen group (defined above)? More generally, can a 1-relator group be SBE to an infinitely presented group?
\end{que}

\subsection{SBE of nilpotent groups}\label{nigr}
This part is the core of the paper. Its purpose is to use our notion of SBE to formulate and improve Goodman's and Pansu's theorems \cite{Goo,Pan}.

Every finitely generated nilpotent group, or more generally any compactly generated, locally compact group has a proper homomorphism with cocompact image into a simply connected nilpotent Lie group, called its real Maltsev completion. Since the completion homomorphism is a quasi-isometry, it is enough to discuss only simply connected nilpotent Lie groups. 

Recall that a Carnot group is a simply connected nilpotent Lie group whose Lie algebra is Carnot. A nilpotent Lie algebra is called Carnot if it admits a Carnot grading, which means a grading by positive integers such that the first layer generates the Lie algebra. 

Let $G$ be a simply connected nilpotent Lie group and $\g$ its Lie algebra. The associated Carnot-graded Lie algebra $\g_\infty$ of $\g$ is the graded Lie algebra $\bigoplus\g^i/\g^{i+1}$, where $(\g^i)_{i\ge 1}$ is the lower central series, and where the bracket $\g^i/\g^{i+1}\times \g^j/\g^{j+1}\to \g^{i+j}/\g^{i+j+1}$ is induced by the restriction of the bracket $\g^i\times\g^j\to\g^{i+j}$. The corresponding simply connected nilpotent Lie group is denoted by $G_\infty$, and called associated Carnot Lie group.

To every finite-dimensional nilpotent Lie algebra $\g$, we associate a numerical invariant $e_\g\in [0,1\mathclose[$, which is practically computable (by solving some linear systems of equations). It satisfies $e_\g=0$ if and only if $\g$ is Carnot. For a $c$-step nilpotent non-Carnot Lie algebra, $e_\g\in\{i/j:2\le i<j\le c\}$ (and all these values can be achieved, see Proposition \ref{alle}). In particular, $e_\g$ belongs to $\{0\}\cup [2c^{-1},1-c^{-1}]$. We write $e_G=e_\g$.

\begin{thm}\label{mainSBE}
Let $G$ be a simply connected nilpotent Lie group with associated Carnot Lie group $G_\infty$, and $e=e_G$. Then $G$ is $O(r^e)$-SBE to $G_\infty$.
\end{thm}

\begin{rem}
Let $c$ be the nilpotency length of $G$. That $G$ is $O(r^{1-c^{-1})}$-SBE to $G_\infty$ was observed by the author in \cite[Proposition A.14]{CoI}\footnote{It appears as Proposition A.9 in the arXiv (v2) version of \cite{CoI}, due to a weird shift of numbering in the published version of the appendix of \cite{CoI}.}, relying on computations performed in a 2007 preprint version of \cite{Bre}; it turns out that the computation already appeared in Goodman's 1977 article \cite{Goo} (which was written before asymptotic cones were defined by Gromov); the exponent is not explicit in the statement of Goodman's theorem but the proof makes clear that it is $\le 1-c^{-1}$. 
 The pioneer reference \cite{Goo} is missing in \cite{Pan,Bre,CoI,BreLD}.
 
The main step towards Theorem \ref{mainSBE} is on the one hand the preliminary algebraic work of \S\ref{wdc}, and on the other hand the estimate of Lemma \ref{lgoodmane}. Goodman \cite[Theorem 1]{Goo} established Lemma \ref{lgoodmane} with the exponent replaced with $e=1-c^{-1}$, in which case the preliminary algebraic work is unnecessary. It is easy to see that this lemma, for a given exponent $e<1$, implies Theorem \ref{mainSBE} for the same exponent $e$: see the proof of Theorem \ref{resbe}. From this and Guivarch's earlier estimates for the distance to the origin, it is easy to conclude
that the asymptotic cone of $G$ is bilipschitz equivalent to $G_\infty$ endowed with a Carnot-Caratheodory metric. This latter step follows conveniently from the SBE language, since the SBE map $G\to G_\infty$ induces a bilipschitz map between asymptotic cones, while it is straightforward that $G_\infty$ endowed with a Carnot-Caratheodory metric is isometric to its asymptotic cones. Pansu in \cite{Pan} was not aware of \cite{Goo} and reproduced a computation akin to that leading to \cite[Theorem 1]{Goo}, but more involved. Indeed, it relies on sharper estimates for the distance to the origin, notably resulting in a description of the asymptotic cone up to isometry and not only up to bilipschitz homeomorphism as above.
The latter remark also applies to the generalization of Pansu's results in the sub-Finsler case by Breuillard and Le Donne \cite{BreLD}.
 
Besides, let us mention that while here we improve the exponent in many cases, the same exponent $1-c^{-1}$ appears in a related finer statement of Breuillard and Le Donne, namely \cite[Proposition 3]{BreLD}\footnote{Proposition 3.1 in the current arXiv version (v1) of \cite{BreLD}.}, for which the exponent $1-c^{-1}=1/2$ in some 2-nilpotent cases is sharp, so that in their statement the exponent $1-c^{-1}$ cannot be replaced by $e_\g$. 
\end{rem}
 
This raises the question about the converse:

\begin{que}\label{q_sbenil}
Given a function $r\mapsto f(r)\ge 1$, is it true that $G$ is $O(f(r))$-SBE to $G_\infty$ if and only if $r^e=O(f(r))$? In particular, is it true that $G$ is $O(r^\alpha)$-SBE to $G_\infty$ if and only if $\alpha\ge e_G$?
\end{que}
 
For $f(r)=1$, a positive answer to the question is equivalent to asking whether being Carnot is a quasi-isometry invariant, or equivalently whether whenever $G$ is non-Carnot, it is not quasi-isometric to $G_\infty$. 
This is the only case where we have a partial answer, as Shalom \cite{sh04} provided a 7-dimensional example of $G$ such that $G$ is not quasi-isometric to $G_\infty$, proving quasi-isometry invariance of the Betti numbers. It was observed in \cite[\S 6.F]{Csu} that Sauer's quasi-isometry invariance of the real cohomology algebra \cite{Sau06}\footnote{The assertion is that whenever $G_1,G_2$ are quasi-isometric simply connected nilpotent Lie groups with Lie algebras $\g_1,\g_2$, then the cohomology algebras $H^*(\g_1)$ and $H^*(\g_2)$ are isomorphic as graded real algebras. Technically speaking, \cite{Sau06} applies only when $G_1,G_2$ admit lattices. To work without restricting to $G_1,G_2$ admitting lattices, the main missing point was to have a ``uniform measure equivalence" for $G_1$ and $G_2$. This was established in \cite[Theorem 5.14]{KKR}, elaborating on \cite[Theorem 1]{BR}, which provides a topological coupling for any two coarsely equivalent second countable locally compact groups. This being granted, Roman Sauer explained to me that the proof works, with some care in the use of continuous cohomology.},  provides further examples, including one in dimension 5 (which is the smallest dimension where non-Carnot $G$ occur). However, even in these examples, we do not know whether, say, $G$ and $G_\infty$ are $O(\log(r))$-SBE, which sounds very unlikely (see Question \ref{dehn}).

These results are obtained in \S\ref{SBEnil}. The algebraic \S\ref{wdc} introduces in detail the Lie algebra conditions necessary to describe $e_\g$, and we hope that its contents will prove relevant in other contexts.

\subsection{Large-scale contractions and similarities} 

In \S\ref{LSCS}, we introduce notions which, unlike the previous ones, are new concepts even in the quasi-isometric setting. Namely, we define a metric space to be large-scale contractable if it admits a self-quasi-isometry in which the large-scale Lipschitz upper bound has multiplicative constant $<1$.  We show that this class is quasi-isometrically closed (Proposition \ref{lsc_qi}); we show that, for connected graphs of bounded valency, it implies polynomial growth (Proposition \ref{lsc_po}). For compactly generated locally compact group, the study thus reduces to simply connected nilpotent Lie groups. Which such groups $G$ are large-scale contractable? Those $G$ with a contracting automorphisms are large-scale contractable; such $G$ can be characterized as those whose Lie algebra admits a grading by positive integers. Are these the only ones? (This is Question \ref{lsc_q}, see there for a little more context.) 

We extend in \S\ref{s_sc} this notion to the sublinearly bilipschitz setting, defining the more general notion of $O(u)$-sublinearly contractible metric space. Understanding for which $u$ a given simply connected nilpotent Lie group is $O(u)$-sublinearly contractible is a challenging problem. An illustrating example, for which partial results can be obtained, is given in Proposition \ref{708}. 

Finally, in \S\ref{s_lss}, we introduce the more subtle notion of large-scale similarity. Let us describe it informally. Given a metric space and a self-map with a contracting behavior, we map a pair $(x,x')$ of points to the number of steps necessary for them to become at distance $\le t$. This yields a function $\Lambda_t(x,x')$. Under reasonable assumption, the class of $\Lambda_t$ modulo addition of a bounded function, does not depend on the choice of large enough $t$. We then say that a self-quasi-isometry is a large-scale $b$-similarity ($b<1$) if this function differs by a bounded function to $\log^+_{b^{-1}}(d(x,x'))$. 

We say that a metric space is large-scale homothetic if it admits a large-scale $b$-similarity for some $b<1$. Somewhat surprisingly, this definition is robust enough to be a quasi-isometry invariant (Corollary \ref{lsh_qi}). Among simply connected nilpotent Lie groups, examples are given by Carnot groups. Indeed, the automorphisms arising from a Carnot grading act by similarities for the Carnot-Caratheodory metric, which is (large-scale) equivalent to the left-invariant Riemannian metrics. It is natural to wonder whether these are the only cases (Question \ref{q_lshcarnot}).

\noindent {\bf Acknowledgement.} I warmly thank Enrico Le Donne for many corrections to a previous version. I am indebted to Gabriel Pallier for his interest to the questions formulated here and for pointing out \cite{Goo}.

\setcounter{tocdepth}{1}
\tableofcontents

\let\addcontentsline=\trueaddcontentsline
\section{Sublinearly Lipschitz maps}\label{s_sbm}

Here, all asymptotic behaviors are meant when $r\to\infty$.

The functions $v$ we consider when defining the $O(r)$-category or $o(r)$-category are assumed to satisfy the following:
\begin{defn}
We say that a function $v:\R_+\to\R$ is admissible if it satisfies:
\begin{itemize}
\item $v$ is non-decreasing;
\item $v$ grows at most linearly: $\limsup_{r\to\infty} v(r)/r<\infty$;
\item $v$ is doubling: $v(tr)/v(r)$ is bounded above for all $t>0$;
\item $v\ge 1$.
\end{itemize}
\end{defn}
However, it is often convenient to work with functions that fail to satisfy one of these hypotheses (or even to be defined) for some small values (e.g., $v(r)=\log(r)^2$ or $v(r)=\log\log r$). Then it should be understood that $v$ is replaced by some function coinciding with $v$ for large $r$ and satisfying the previous assumptions.

Note that $v(t+c)/v(t)$ is also bounded above for all $c$, because $t+c\le 2t$ for large $t$ and the doubling assumption.

\begin{prop}
Let $v$ be an admissible map.
A composition of $O(v)$-Lipschitz maps is $O(v)$-Lipschitz. Moreover, composition is compatible with $O(v)$-closeness. The same assertions hold with $o(v)$ instead of $O(v)$.
\end{prop}
\begin{proof}
Suppose that $f:X\to Y$ and $g:Y\to Z$ are $O(v)$-Lipschitz, say with constants $C,C'$ in the metric upper bound; also we multiply $v$ by a constant if necessary to remove the other constant and assume $v\ge 1$. Then, denoting by $x_0$ the fixed base-point, we have $|f(x)|\le C|x|+v(|x|)$ for all $x$. 
Also there exists $M,\mu$ such that $v(r)\le \mu r+M$ for all $r$. So $|f(x)|\le (C+\mu)|x|+M$. 

\begin{align*}
d(g\circ f(x),g\circ f(x')) \le & C'd(f(x),f(x'))+v(|f(x)|+|f(x')|)\\
	\le & C'Cd(x,x')+C'v(|x|+|x'|)+v((C+\mu)(|x|+|x'|)+2M)\\
\end{align*}
There exists $\lambda>0$ such that $v((C+\mu)r+2M)\le \lambda v(r)$ for all $r$. Therefore, for all $x,x'\in X$,
\[d(g\circ f(x),g\circ f(x')) \le C'Cd(x,x')+(C'+\lambda)v(|x|+|x'|).\]
Thus $g\circ f$ is $O(v)$-Lipschitz (with Lipschitz constant $CC'$).

Let us now prove the second assertion. Assume that $g,g'$ are $O(v)$-close, say $d(g(y),g'(y))\le v(|y|)$ and let us show that $g\circ f$ and $g'\circ f$ are $O(v)$-close. Then

\[d(g\circ f(x),g'\circ f(x)) \le  v(|f(x)|)
	\le  v((C+\mu)|x|+M)\le\lambda v(|x|)\]

Now assume that $f,f':X\to Y$ are $O(v)$-close, say $d(f(x),f'(x))\le v(|x|)$ (as we can assume multiplying $v$ by a positive constant if necessary), and let us show that $g\circ f$ and $g\circ f'$ are $O(v)$-close.

Then 
\begin{align*}
d(g\circ f(x),g\circ f'(x)) \le & C'd(f(x),f'(x))+v(|f(x)|+|f'(x)|)\\
	\le & C'v(|x|)+v(2(C+\mu)|x|+2M)\le (C'+\lambda')v(|x|),
\end{align*}
where $\lambda'$ is chosen so that $v(2(C+\mu)r+2M)\le\lambda'v(r)$ for all $r$.

The statements for $o(v)$ follow immediately.\end{proof}

\begin{defn}
Let $v$ be an admissible map. We say that a map $f:X\to Y$ is
\begin{itemize}
\item $O(v)$-expansive if there exist constants $C,C'>0$ such that $d(f(x),f(x'))\ge Cd(x,x')-C'v(|x|+|x'|)$ for all $x,x'\in X$;
\item $O(v)$-surjective if there exists a constant $C''>0$ such that $d(y,f(X))\le C''v(|y|)$ for all $y\in Y$. 
\item $o(v)$-expansive, respectively $o(v)$-surjective if there exists $u=o(v)$ such that $f$ is $O(u)$-expansive, resp.\ $O(u)$-surjective.
\end{itemize}
(By convention, $f$ is $O(v)$-expansive and $o(v)$-expansive if $X=\emptyset$, and the map $\emptyset\to Y$ is $O(v)$-surjective, resp.\ $o(v)$-surjective, if and only if $Y=\emptyset$.)
\end{defn}

\begin{prop}
Let $v$ be an admissible map. A map $f:X\to Y$ induces an isomorphism in the $O(v)$-category if and only if it satisfies:
\begin{itemize}
\item $f$ is $O(v)$-Lipschitz;
\item $f$ is $O(v)$-expansive;
\item $f$ is $O(v)$-surjective.
\end{itemize}
The same statement holds with $O(v)$ replaced with $o(v)$.
\end{prop}
\begin{proof}
If $f$ induces an isomorphism in the $O(v)$-category, with inverse induced by $g:Y\to X$, all conditions are straightforward to check.

Conversely, suppose that $f$ satisfies all these conditions. Leaving the trivial case $X=\emptyset$ to the reader, we assume the contrary, so $Y\neq\emptyset$ as well. Besides, we can suppose $v\ge 1$, $C''=1/2$ and $C'=1$, up to changing $v$ to $\max(v,1)$ and then multiplying $v$ by a constant $\ge 1$. For each $y\in Y$, choose $g(y)\in X$ such that $d(f(g(y)),y)\le v(|y|)$.

By the second condition, we have $|f(x)|\ge C|x|-v(|x|)$ for all $x\in X$. There exists a constant $k$ such that $Cr-v(r)\ge Cr/2-k$ for all $r$. Hence $|f(g(y)|\ge C|g(y)|/2-k$, that is, $C|g(y)|/2\le k+|f(g(y)|\le k+|y|+v(|y|)$. Hence $|g(y)|\le c|y|+k'$ for suitable constants $c,k'>0$ and all $y\in Y$; in turn there exists a constant $c'>0$ such that $v(2cr+2k')\le c'v(r)$ for all $r$.

Again by the second condition, for all $y,y'\in Y$ we have
\[d(f(g(y)),f(g(y')))\ge Cd(g(y),g(y'))-v(|g(y)|+|g(y')|).\]
So 
\begin{align*}
Cd(g(y),g(y'))\le &  d(f(g(y)),f(g(y')))+v(2c|y|+2k')\\
\le & d(y,y')+v(|y|)+v(|y'|)+v(2c|y|+2k')\\
\le & d(y,y')+3v(|y|+|y'|).
\end{align*}

This shows that $g$ is $O(v)$-Lipschitz. It is clear by construction that $f\circ g$ is $O(r)$-close to the identity of $Y$. By the second condition,

\[d(f(g(f(x))),f(x))\ge Cd(g(f(x)),x)-v(|x|+|g(f(x))|)\]
In addition, we have
$d(f(g(f(x))),f(x))\le v(|f(x)|)\le c''v(|x|)$ for some suitable constant $c''$, and $|g(f(x))|=O(|x|)$ as well, se we can deduce that $Cd(g(f(x)),x)=O(v(|x|))$.

The $o(v)$-case follows (from the $O(v')$-case for all $v'$).
\end{proof}

\begin{defn}
An $O(v)$-Lipschitz map $g:Y\to X$ is an $O(v)$-retract if it is a retract in the $O(v)$-category, that is, there exists an $O(v)$-Lipschitz map $f:X\to Y$ such that $g\circ f$ is $O(v)$-close to the identity of $X$. Similarly one defines $o(v)$-retract, and $o(r)$-retract is also referred to sublinearly Lipschitz retract.
\end{defn}

\begin{ex}
Let $G$ be a simply connected solvable Lie group, and $E$ its exponential radical (the set of exponentially distorted elements in $G$). Then $G/E$ is an SBE-retract of $G$, by \cite[Theorem 4.4]{CoI}; in substance this was the main argument in the dimension estimate of the preceding paper \cite{CoJT}. Note that this is only interesting in the non-split case, since $G/E$ is indeed a Lipschitz retract when $G\sim E\rtimes (G/E)$.
\end{ex}

\section{Growth}\label{growth_ame}

Let us say that a metric space is uniformly locally finite (ULF) if the supremum of all cardinals of $n$-balls in $X$ is bounded. This is a mild reasonable assumption to consider growth and amenability conditions. Connected graphs of bounded degree are ULF.

\begin{ex}\label{counter_tre}
In a reasonable generality (e.g., connected graphs of bounded valency), the asymptotics of growth and amenability are quasi-isometric invariants. We begin with two simple counterexamples to discard naive generalizations to the SBE-setting and motivate the sequel. Fix any sublinear function $u$.
\begin{enumerate}
\item\label{polgr_ni}
For connected graphs of bounded valency, polynomial growth is not a SBE-invariant: for instance, construct a tree $T$ from a ray $R=\{x_n:n\ge 0\}$ by attaching to each $x_n$ a rooted 2-tree $T_n$ of depth $u(n)$. The map retracting $T$ to $R$ mapping $T_n$ to $x_n$ is an SBE (an $O(u)$-SBE) $T\to R$ and $R$ has polynomial (linear) growth. If $\lim_{n\to\infty}u(n)/\log(n)=\infty$, then $2^{u(n)}$ grows superpolynomially and hence $T_n$ does not have polynomially bounded growth: it can have arbitrary large subexponential growth.
\item\label{camesbi} For connected graphs of bounded valency, amenability is not an SBE-invariant. Perform a very similar construction as in (\ref{polgr_ni}), but instead of attaching bushy small trees to a infinite thin tree, we attach filiform trees to a infinite bushy tree. Precisely, we perform the same construction, but start with an infinite rooted tree $R'$ of valency 3, and attaching at every vertex of height $n$ a ray of length $u(n)$. Then the resulting tree $T'$ is amenable, its obvious retraction to $R'$ is an SBE (an $O(u)$-SBE), and $R'$ is not amenable.
\end{enumerate}
\end{ex}

Let us say that a pointed metric space is doubling at large scale if there exists $M$ and $R_0$ such that for every $R\ge R_0$ every closed centered $2R$-ball is contained in the union of $M$ closed $R$-balls.
I expect the following fact to be well-known:

\begin{prop}
A pointed metric space is doubling at large scale if and only if all its asymptotic cones (on the basepoint) are proper.
\end{prop}
\begin{proof}
In the definition, say that the pointed metric space $(X,d)$ ($d$ being the distance) is $(R_0,M)$-doubling. If so, this implies that $(X,\frac1{n}d)$ is $(R_0/n,M)$-doubling. It follows that every asymptotic cone of $X$ is $(\eps,M)$ doubling for every $\eps$, which is the definition of being doubling (as a pointed metric space). A complete doubling pointed metric space is proper: this follows from the fact that a complete metric space is compact if and only if for every $\eps>0$ it can be covered by finitely many $\eps$-balls \cite[\S 2.7, Prop.\ 10]{Bou}. Since asymptotic cones are complete, this proves one implication.

Conversely suppose that $X$ is not large-scale doubling as a pointed metric space. Then for all $n$ there exists some centered ball of arbitrary large radius, say of radius $2R_n\ge n$, that cannot be covered by less that $n$ balls of radius $R_n$. Hence in $(X,\frac{1}{R_n}d)$, the centered 2-ball cannot be covered by less that $n$ balls of radius 1. Thus it contains $n$ points at pairwise distance $\ge 1$. It follows that for any non-principal ultrafilter, the cone for the scaling sequence $(1/R_n)$ over this ultrafilter, has infinitely (actually at least continuum) many points that are at pairwise distance $\ge 1$; hence it is not proper.
\end{proof}

Since SBE spaces have bilipschitz equivalent asymptotic cones, the following corollary follows.

\begin{cor}Being large-scale doubling (as pointed metric spaces) is an SBE-invariant.\end{cor}

\begin{rem}
For a uniformly locally finite space (or more generally, a uniformly coarsely proper metric space, with a suitable notion of growth, see \cite[\S 3.D]{CH}), large-scale doubling (as pointed metric space) implies polynomially bounded growth. The converse is not true in general. For instance, consider a rooted tree in which the vertices in the $k$-sphere have 2 successors for $2^{n^2}-n<k<2^{n^2}$ and all integers $n$, while all others vertices have a single successor. Then it has polynomially bounded growth, but is not large-scale doubling.
\end{rem}

However a compactly generated, locally compact group is large-scale doubling if and only if it has polynomially bounded growth (the forward implication is immediate but the reverse implication makes use of deep theorems, including Gromov's polynomial growth). So we deduce:

\begin{cor}
Among compactly generated locally compact groups, and hence among finitely generated groups or transitive connected locally finite graphs, polynomial growth is an SBE-invariant.
\end{cor}

Given a pointed metric space $X$, let $b_X(n)$ be the cardinal of the closed centered ball of radius $n$, and $b^\mathrm{u}_X(n)$ the supremum of cardinals of all closed balls of radius $n$.

We say that $X$ has uniform exponentially bounded (UEB) growth if $\overline{\lim} b^\mathrm{u}_X(n)^{1/n}=1$. This is a mild assumption, as this holds if $X$ is a connected graph of bounded valency.

\begin{prop}\label{sub_sbe}
Let $X,Y$ be metric spaces with basepoints. Let $u$ be a non-decreasing sublinear function and $f:X\to Y$ be a map such that for some $R_0,C>0$, we have
\[|f(x)|\le \max(|x|,R_0),\quad d(f(x),f(x'))\ge Cd(x,x')-u(|x|\vee |x'|),\quad\forall x,x'\in X.\]

Then for all $n\ge R_0$ we have $b_Y(n)\ge b_X(n)/b_X(u(n)/C)$.

In particular, if $X$ has UEB growth and $Y$ has subexponential growth, then $X$ also has subexponential growth. In particular, among UEB growth metric spaces (e.g., connected graphs of bounded valency), having subexponential growth is an SBE-invariant.
\end{prop}
\begin{proof}
Fix $n\ge R_0$. Then the $n$-ball of $X$ maps into the $n$-ball of $Y$. If $x,x'$ belong to the $n$-ball and have the same image, then the second inequality implies that $Cd(x,x')\le u(n)$. In particular, $f(x)=f(x')$ implies that $x'$ belongs to the ball of radius $u(n)/C$ around $x$. So the cardinal of the fibers of $f$ contained in the $n$-ball is $\le b^\mathrm{u}_X(u(n)/C)$. Thus the cardinal of the image of the $n$-ball is $\ge b_X(n)/b^\mathrm{u}_X(u(n)/C)$.

If $X$ has UEB growth, write $b_X^\mathrm{u}(n)\le\exp(\alpha n)$ (say for $n\ge R_1\ge R_0$). Then we deduce that $b_X(n)\le\exp(\alpha u(n)/C)b_Y(n)$, where the right-hand term grows subexponentially.
\end{proof}

\begin{cor}\label{sub_SBE}
For finitely generated groups (and more generally, compactly generated locally compact groups) subexponential growth is an SBE-invariant.\qed
\end{cor}

\begin{que}
Does there exist two SBE finitely generated groups (or, more generally, compactly generated locally compact groups) that have non-equivalent growth? Such groups would necessarily be of intermediate growth.
\end{que}

Proposition \ref{sub_sbe} also yields some quantitative statements:
\begin{cor}
1) The class of UEB growth metric spaces of polynomial growth is $O(\log r)$-SBE-invariant. 

2) For every $\alpha>1$, the class of UEB growth metric spaces (and hence of groups) whose growth is $O(\exp(r^{\alpha}))$ is $O(r^{\alpha})$-SBE-invariant. In particular, the union of all these classes (for $\alpha<1$) is $O(r^{1-\eps})$-SBE-invariant for all $\eps>1$.\qed
\end{cor}

Amenability is not an SBE-invariant, nor even an $O(u)$-SBE invariant of connected graphs of bounded degree, for any function $u$ tending to infinity.

\begin{que}\label{ame_SBE}
Is amenability a SBE-invariant of finitely generated groups (and more generally unimodular compactly generated locally compact groups)?
\end{que}

(Recall that in the nondiscrete case, the good setting is that of amenable unimodular groups, also called metrically amenable groups, see \cite[\S 3.D]{CH}.)

In case of a negative answer to Question \ref{ame_SBE} one can wonder whether there are strengthenings of amenability that hold for natural instances (e.g., all amenable unimodular connected Lie groups) that are SBE-invariants. A partial answer is given by the class of groups with subexponential growth. 
But it would be interesting to have a result encompassing some amenable groups of exponential growth such as polycyclic groups, which have good quantitative amenability properties.


\section{Ends and Hyperbolicity}

\subsection{Ends of spaces}\label{ss_ends}

Fix a geodesic metric space $X$. Assuming that $X$ is non-empty, fix a base-point. We consider the set of proper paths $(x_t)$ in $X$. Here proper means that $|x_t|$, the distance of $x_t$ to the base-point, tends to infinity, and path just means that $x_t$ depends continuously on $t$.

The forking between two such paths $(x_t)$ and $(x'_t)$ is the number $\phi((x_t),(x'_t))\in\R_+\cup\{\infty\}$ defined as the supremum of $r$ such that for all large $t$, there exists a path from $x_t$ to $x'_t$ avoiding the open $r$-ball (``for all large" can be replaced with ``for arbitrary large" without modifying the definition). It satisfies the ultrametric-like inequality $\phi(c,c'')\ge\min(\phi(c,c'),\phi(c',c''))$ for all proper paths $c,c',c''$. Thus the function $\delta(c,c')=2^{-\phi(c,c')}$ is an ultrametric pseudo-distance on proper Lipschitz paths. The Hausdorff quotient is an ultrametric space, called space of ends of $X$ and denoted by $E(X)$. It is easy to see that its distance depends on the choice of base-point only up to bilipschitz. If $X$ is proper, $E(X)$ is a compact space. When $X$ is empty, we define $E(X)$ to be empty, although the empty space will play an empty role here!

Note that in particular if two proper paths $(x_t)$ and $(x'_t)$ coincide arbitrary far (the set of $(t,t')$ such that $x_t=x'_{t'}$ is unbounded) then they define the same element of $E(X)$.

\begin{proof}[Proof of Theorem \ref{ends_func}]
Consider $f:X\to Y$ as in the theorem, satisfying \[d(f(x),f(x'))\le Cd(x,x')+v(|x|+|x'|)\] for all $x,x'$, and $|x|\le |f(x)|\le C|x|$ for all $x$. The latter (instead of $C'|x|-C''\le |f(x)|\le C|x|+C'''$) is a mild additional assumption in order to ease reading the proof; the reader can easily adapt to the general case. We also assume that $v\ge 1$ and $v$ is non-decreasing.

Let $(x_t)$ be a proper path in $X$. Let $(t_n)$ be a sequence tending to infinity (called discretization) such that $d(x_{t_n},x_{t})\le 1$ for all $t\in [t_n,t_{n+1}]$ and all $n$. Define $y_n=f(x_{t_n})$; then $d(y_n,y_{n+1})\le C+v(2|x_n|)\ll |y_n|$. Thus if we interpolate, using geodesic segment, $(y_n)$ to a path $(z_t)$, then $(z_t)$ is proper, and its class in $E(Y)$ does not depend on the choice of interpolation. Furthermore, if we use a finer discretization, then the class in $E(Y)$ of the resulting path remains the same. So we have a well-defined map from the set of proper paths to $E(Y)$.

There exists $R_0$ such that $v(r)\le r/2$ for all $r\ge R_0$. Fix $R\ge R_0$. Let $(x_t)$ and $(x'_t)$ be proper paths with $\phi((x_t),(x'_t))\ge 4R$. Define $(y'_t)$ and $(z'_t)$ in the same way as above.

Use a discretization $(t_n)$ working for both paths. There exists $t_0$ such that for all $t\ge t_0$, there exists a path from $x_t$ to $x'_t$ outside the open $2R$-ball of $X$. Fix $t\ge t_0$ and such a path $\gamma_t$. We can assume that $\gamma_t$ is defined on an interval $[0,m]$ with $m$ an integer, and that $d(\gamma_t,\gamma_u)\le 1$ for all $t,u$ such that $|t-u|\le 1$. Define, for $n\in\{0,\dots,m\}$, $\mu_n=f(\gamma_n)$. Then $|\mu_n|\ge |\gamma_n|\ge 2R$, and we have
\[d(\mu_n,\mu_{n+1})\le C+v(|\gamma_n|+|\gamma_{n+1}|)\le C+v(2|\gamma_n|+1)\le C+v(4R+1)\le R.\]
Thus $|\mu_n|-d(\mu_n,\mu_{n+1})\ge 2R-R=R$. So, if we interpolate $(\mu_n)$ to a path $(\mu_t)$ using geodesic segments, we have $|\mu_t|\ge R$ for all $R$, and this joins $y_t$ to $y'_t$ outside the open $R$-ball. Thus $\varphi((z_t),(z'_t))\ge R$. This shows that whenever $r\le 2^{-R_0}$, we have $\delta((x_t),(x'_t))\le r^4$ implies $\delta((z_t),(z'_t))\le r$. This shows that the map $(x_t)\mapsto (z_t)$ factors through a H\"older continuous map $E(X)\to E(Y)$. (Note however that the H\"older exponent 1/4 is an artifact of the assumption $|f(x)|\ge |x|$ and we cannot expect it to be bounded away from 0 in general.)
\end{proof}

\subsection{Hyperbolic spaces}\label{s_hyp}

Gromov-hyperbolicity is defined for arbitrary metric spaces; see the definition below in the proof of Lemma \ref{ghlemma}. A fundamental property is its quasi-isometry invariance among geodesic metric spaces \cite[\S 5.2]{GH}, where it is also observed that there exists a (non-geodesic) metric space quasi-isometric to $\Z$ but not Gromov-hyperbolic.

\begin{ex}\label{sbe_noninv}
Let $u$ be any sublinear function, say mapping positive integers to positive integers. Decorate the graph $\Z$ by adding for all $n$, between $2^n$ and $2^n+u(n)$, a second branch of size $u(n)$; let $X$ be the resulting graph. Then the embedding $\Z\to X$ is an isometric SBE, but $X$ is not Gromov-hyperbolic (nor even large-scale simply connected).
\end{ex}

This shows that among geodesic metric spaces, Gromov-hyper\-bolicity is not an SBE-invariant. However, by the argument explained in \S\ref{intro_hyp}, SBE-invariance of Gromov-hyper\-bolicity holds under a homogeneity assumption:

\begin{defn}\label{qiho}
A metric space is quasi-isometrically homogeneous (or QI-homo\-geneous) if quasi-isometries with uniform constants act transitively. Namely, this means that there exist $C\ge 0$ and $c\ge 1$ such that for any $x_0,x\in X$, there exists $f:X\to X$ such that $f(x)=x_0$ and $f$ is a quasi-isometry for the constants $c,C$: $d(f(x),f(x'))$ is in the interval $[c^{-1}d(x,x')-C,cd(x,x')+C]$ for all $x,x'\in X$ and $\sup_{y\in X}d(y,f(X))\le C$. 
\end{defn}

\begin{prop}\label{hyp_sbe_inv}
Among QI-homogeneous (Definition \ref{qiho}) geodesic metric spaces, Gromov-hyperbolicity is an SBE-invariant.\qed
\end{prop}

We now turn to Theorem \ref{i_hyp_bou}. We restate it slightly more precisely:

\begin{thm}\label{hyp_bou}
Let $f:X\to Y$ be a sublinearly Lipschitz map between proper geodesic $\delta$-hyperbolic spaces. Assume $f$ is radially expansive with multiplicative constant $c>0$, in the sense that $|f(x)|\ge c|x|-c'$ for some $c'\in\R$ and all $x\in X$. Then, for every $\mu\in\mathopen]0,2^{1/2\delta}]$, every $\eps>0$, the map $f$ induces a $(c-\eps)$-H\"older map $(\partial X,d_\mu)\to (\partial Y,d_\mu)$ on the visual boundaries. In particular, any  SBE $X\to Y$ induces a H\"older homeomorphism between the boundaries.
\end{thm}

Here $d_\mu$ is the metric on the boundary, depending on the parameter $\mu\in\mathopen]0,2^{1/2\delta}]$ (its definition is recalled in the proof below).
Note that the H\"older exponent $c-\eps$ does not depend on the constants and function involved in the definition of being sublinearly Lipschitz, but only on the ``radial" expansion $c$.

Write $\R_+=[0,\infty\mathclose[$. Let $X$ be a set and $\rho:X\times X\to\R_+$ be a function (this is called a kernel on $X$). We say that $\rho$ is subadditive if $\rho(x,z)\le\rho(x,y)+\rho(y,z)$ for all $x,y,z$. For every kernel there is a largest subadditive kernel $\hat{\rho}:X\times X\to\R_+$ bounded above by $\rho$; $\hat{\rho}(x,y)$ is the infimum over all $n$ and all sequences $x=x_0,x_1,\dots,x_n=y$ of $\sum_{i=1}^n\rho(x_{i-1},x_i)$. Of course $\hat{\rho}\le\rho$; if $\rho$ is symmetric then so is $\hat{\rho}$, and if $\rho$ vanishes on the diagonal then so does $\hat{\rho}$.

\begin{lem}[Ghys-Harpe]\label{ghlemma}
Let $\rho:X\times X\to\R_+$ be a kernel. Suppose that for some $\lambda\in [1,\sqrt{2}]$ we have, for all $x,y,z$, $\rho(x,z)\le\lambda\max(\rho(x,y),\rho(y,z))$. Then $(3-2\lambda)\rho\le\hat{\rho}$.
\end{lem}
\begin{proof}[On the proof]
The statement is originally stated in a more specific context with the change of variables $\lambda=1+\eps'$. It is proved in \cite[Chap 7,\S 3]{GH} and the statement and proof is almost textually rewritten in \cite[Chap III.H, \S 3.21]{BH}; it provides with no change the above statement.
\end{proof}

Let us now prove Theorem \ref{hyp_bou}. We assume $c'=0$ since we can easily reduce to this case, or adapt the proof.

Let $Y$ be a pointed metric space. Recall that the Gromov product is defined by 
$2(x|y)=|x|+|y|-d(x,y)$. The pointed\footnote{One usually says that a metric space is $\delta$-hyperbolic if the pointed metric space $(Y,y)$ is $\delta$-hyperbolic for every $y$. Actually, if $(Y,y)$ is $\delta$-hyperbolic, then $Y$ is $2\delta$-hyperbolic \cite[Chap.\ 1]{CDP}.} metric space $Y$ is $\delta$-hyperbolic if $(x|z)\ge\min\{(x|y),(y|z)\}-\delta$ for all $x,y,z$. 

Let us briefly recall the definition of boundary. A sequence $(x_n)$ in a geodesic Gromov-hyperbolic space $Y$ is Cauchy-Gromov if it satisfies $(x_n|x_m)\to\infty$ when $\min(n,m)\to\infty$. Two Cauchy-Gromov sequences $(x_n),(y_n)$ are asymptotic if $(x_n|y_n)\to\infty$ when $n\to\infty$; this is an equivalence relation on the set of Cauchy-Gromov sequences, and the quotient is called the boundary of $Y$, and denoted by $\partial Y$. If $Y$ is proper, every Cauchy-Gromov sequence is asymptotic to a sequence forming a geodesic ray (an isometric embedding of $\N$).

For the moment however, we work within $Y$ rather than its boundary.
For $\mu>0$, define $\rho_\mu(x,y)=\mu^{-(x|y)}$. If $Y$ is $\delta$-hyperbolic, then $\rho_\mu$ satisfies the inequality of Lemma \ref{ghlemma} with $\lambda=\mu^\delta$; thus the lemma applies as soon as $\mu^\delta\le\sqrt{2}$. Actually, we will assume (to simplify constants) that $3-2\lambda\ge 1/2$, that is, $\lambda\le 5/4$, which means $\mu\le (5/4)^{1/\delta}$. We set $d_\mu=\widehat{\rho_\mu}$, so that $d_\mu\le\rho_\mu\le 2d_\mu$.

Let $(y_n)$ be a (possibly finite) sequence in $Y$ such that
\[|y_n|\ge cn;\quad d(y_n,y_{n+1})\le u(n),\quad u(n)\ll n,\quad u\text{ non-decreasing}.\]

For every $\eps>0$, there exists $k_\eps$ such that $u(n)\le 2\eps n$ for all $n\ge k_\eps$.

Then, for $k_\eps\le i$, we have
\[(y_i|y_{i+1})=\frac12(|y_i|+|y_{i+1}|-d(y_i,y_{i+1}))\ge ci-\frac12u(i)\ge (c-\eps)i.\]
Hence
\[d_\mu(y_i,y_{i+1})\le \mu^{-(c-\eps)i}\]
Then, for $k_\eps\le k\le n$, denoting $c_\eps=c-\eps$, we have
\[d_\mu(y_k,y_n)\le\sum_{i=k}^{n-1}d_\mu(y_i,y_{i+1})\le\sum_{i=k}^{n-1}\mu^{-c_\eps i}
\le\frac{\mu^{-c_\eps k}-\mu^{-c_\eps n}}{1-\mu^{-c_\eps}}\le\frac{\mu^{-c_\eps k}}{1-\mu^{-c}},\]
and thus 
\[(y_k|y_n)=-\log_\mu(\rho_\mu(y_k,y_n))\ge-\log_\mu(2d_\mu(y_k,y_n))\]
\[\ge c_\eps k-\log_\mu(2/(1-\mu^{-c}))=:c_\eps k-M(c,\mu)\]

Now let $(y_n)$ and $(y'_n)$ be sequences satisfying the same conditions, as well as $d(y_j,y'_j)\le u(j)$ for all $j\le k$, for some given $k$. For each $\eps$ such that $k\ge k_\eps$ and $n\ge k$, we have
\[(y_k|y_n),(y'_k|y'_n)\ge c_\eps k-M(c,\mu);\]
in the same way we bounded below $(y_i|y_{i+1})$, we have
\[(y_k|y'_k)\ge c_\eps k.\]
So 
\[(y_n|y'_n)\ge\min\{(y_n|y_k),(y_k|y'_k),(y'_k|y'_n)\}-2\delta\ge c_\eps k-M(c,\mu)-2\delta.\]

Now let $X,Y$ be $\delta$-hyperbolic proper geodesic spaces, and $f$ a map $X\to Y$ satisfying
\[f(x)\ge c|x|;\quad d(x,x')\le 2\delta\Rightarrow d(f(x),f(x'))\le u(|x|)\] 

Let $x,x'$ be points in $X$. Join them with geodesics to the origin $(x_i)_{0\le i\le n}$, $(x'_i)_{0\le j\le m}$. Define $k=\lfloor (x|x')\rfloor$. Then $d(x_j,x'_j)\le 2\delta$ for all $j\le k$.

Define $y_i=f(x_i)$, $y'_i=f(x'_i)$. Then the previous argument shows that $(y_n|y'_m)\ge c_\eps k-M(c,\mu)-2\delta$, as soon as $k\ge k_\eps$. That is, 
\[\forall x,x',\;(x|x')> k_\eps\quad\Rightarrow\quad(f(x)|f(x'))\ge c_\eps (x|x')-N(c,\mu,\delta),\]
where $N(c,\mu,\delta)=c/2+M(c,\mu)+2\delta$.

An immediate consequence is that $f$ maps Cauchy-Gromov sequences to Cauchy-Gromov sequences, and thus extends to a map $f_*$ between boundaries. This can be made more quantitative.

Recall that the Gromov product extends to the boundary, with some finite additive ambiguity. Precisely, one can define \cite[\S 7.2]{GH}, for $\omega,\omega'\in\partial X$, their Gromov product $(\omega|\omega')$ as the supremum of $\liminf_{i,j\to\infty}(x_i|x'_j)$ where $(x_i)$, $(x'_i)$ range over all sequences representing $\omega$ and $\omega'$ respectively (i.e., converging to $\omega$ and $\omega'$).

 Let $\omega,\omega'$ be boundary points. Choose geodesic rays $(x_n)$, $(x'_n)$ converging to them. Then $\liminf(x_n|x'_n)-(\omega|\omega')\in [-2\delta,0]$ \cite[Chap 7.1]{GH}, and $\liminf(f(x_n)|f(x'_n))-(f_*(\omega)|f_*(\omega'))\in [-2\delta,0]$.
Hence, if $(\omega|\omega')>k_\eps+2\delta$, then
\[(f_*(\omega)|f_*(\omega'))\ge c_\eps(\omega|\omega')-2\delta c-N(c,\mu,\delta)=c_\eps(\omega|\omega')-N'(c,\mu,\delta)\]

So for all $\omega,\omega'$ such that $d_\mu(\omega,\omega')<\mu^{-k_\eps-2\delta}/2$, we have 
\[\rho_\mu(f_*(\omega),f_*(\omega'))\le P(c,\mu,\delta)\rho_\mu(\omega,\omega')^{c_\eps},\quad P=\mu^{N},\]
so that 
\[d_\mu(f_*(\omega),f_*(\omega'))\le 2^cP(c,\mu,\delta)d_\mu(\omega,\omega')^{c_\eps}.\]

This shows that $f_*$ is $c$-H\"older. (If we define $Q=\min(2^cP,(2\mu^{k_\eps+2\delta})^c)$, then the inequality $d_\mu(f_*(\omega),f_*(\omega'))\le Qd_\mu(\omega,\omega')^{c_\eps}$ holds for all $\omega,\omega'$.)

\section{SBE and spaces with linear isodiametric function}\label{SBEfp}

In an attempt to determine whether finite presentability is an SBE-invariant of finitely generated groups, it has been natural to focus on the more restricted class of groups with linear isodiametric function, which contains most interesting finitely presented groups. 

Here we define this notion in the general setting of connected graphs. 

\begin{defn}
In a connected graph with base-point, we say that two combinatorial based loops are $R$-homotopic if they are homotopic as based loops in the 2-complex obtained by gluing $k$-gons to all loops of length $k\le R$. 

We say that a connected graph has the linear isodiametric filling property (abbreviated LID) if it satisfies the following: given a base-point, there exists $R,C$ such that for every $n$, every loop of diameter $\le n$ is $R$-homotopic to a trivial loop within a ball of radius $\le Cn$. (By convention the empty space is LID.)

We say that a compactly generated locally compact group is LID if it quasi-isometric to some LID connected graph.
\end{defn}

Note that the choice of base-point does not matter in the definition of LID. Also, we use the notion of QI-homogeneous from Definition \ref{qiho}.

\begin{prop}
Being LID is an SBE-invariant among QI-homogeneous graphs. 

More generally, suppose that there is a sublinearly Lipschitz retract between connected graphs $Y\to X$, that $X$ is QI-homogeneous and $Y$ is LID. Then $X$ is LID.
\end{prop}
\begin{proof}
Assume that $f,g$ are $O(u)$-Lipschitz with $u(r)=o(r)$, and that $f\circ g$ is $O(u)$-close to the identity (we multiply $u$ by a constant if necessary to eliminate the constants multiplying $u$ in the definition). Consider a combinatorial loop $(x_i)$ in the $n$-ball of $X$. Then $d(x_i,g\circ f(x_i))\le u(|x_i|)\le u(n)$.

Define $y_i=f(x_i)$. We can interpolate $(y_i)$ to a path $(z_j)$ in the $Cn$-ball, where $C'$ only depends on $f$. Then since $Y$ is LID (with constant $C$ in the definition) we can homotope $(z_j)$ to a trivial loop within the $CC'$-ball of $Y$. Here, homotope means that we can find a finite sequence of loops in the $CC'$-ball of $Y$, each of which differs from the next one by an $R$-gon. Now we can map this sequence of loops to $X$ using $g$; we obtain a sequence of ``paths" which can be interpolated using geodesic segments; moreover, any two of them differ by the ``image of the $R$-gon", which is a closed path of length $\le Ru(CC'n)$.

Finally, we can pass from $(x_i)$ to the path interpolating $(g\circ f(x_i))$  by filling the ``squares" with vertices $x_i$, $x_{i+1}$, $g\circ f(x_i)$, $g\circ f(x_{i+1}))$, which can are interpolated by gons of length $1+2u(n)+w$, where $w$ comes from interpolating in $X$ the interpolated path in $Y$ between $f(x_i)$ and $f(x_{i+1})$. In $Y$, this path (now $i$ is fixed) has length $\le u(n)$, and hence in $X$ (with one more interpolation) it has length $\le u(n)^2$, which is somewhat a problem (unless $u(n)^2=o(n)$). But actually for every $z$ in the interpolated path between $f(x_i)$ and $f(x_{i+1})$, we have $d(z,f(x_i))\le u(n)$. So $d(g(z),g(f(x_i)))\le (c+1)u(n)$, where $c$ is the Lipschitz constant in $g$. So we can fill our large gon using segment between $g(f(x_i)$ and images of the interpolated gon in $Y$. We thus obtain $u(n)$ ``triangular" gons of length $\le 2(c+1)u(n)+u(n)=(2c+3)u(n)$ (containing $g(f(x_i)$, $g(z_j)$, $g(z_{j+1})$ for various $j$), and one ``square" gon of length $\le 3u(n)+(c+1)u(n)=(c+4)u(n)$ (containing $x_i$, $x_{i+1}$, $g(f(x_i))$, $g(f(x_{i+1}))$). 

First assume for simplicity that $X$ is homogeneous under isometries. Then it follows from the previous argument, assuming that $\max((2c+4)u(n),Ru(CC'n))<n$, which holds, say for $n\ge n_0$, we can fill every based loop in the $n$-ball with loops of smaller radius. These being translates of based loops, we can conclude by induction. In the QI-homogeneous case, there is some loss due to interpolation, but if we assume that $\max((2c+4)u(n),Ru(CC'n))<\eps n$ with $\eps>0$ small enough, depending on the constants involved in QI-homogeneity of $X$, then we can conclude in the same way using easy interpolation arguments. 
\end{proof}

\begin{cor}Among compactly generated locally compact groups (and in particular among finitely generated groups), LID is an SBE-invariant property and more generally is inherited by sublinearly Lipschitz retracts. 
\end{cor}

\section{Lie algebras and SBEs of nilpotent groups}\label{SBEnil}

This section is more elaborate than the previous ones. While our goal is the study of SBE between nilpotent groups (\S\ref{su_nil}), it is convenient to start with a solid algebraic preparatory work, in \S\ref{wdc}. This material is new and may also be of independent interest. To provide a warm-up and some motivation, we start with the particular cases of 3-step and 4-step nilpotent Lie groups and Lie algebras in \S\ref{part3} and \S\ref{part4}; these can also be omitted by the reader not looking for specific motivation.

\subsection{The 3-step nilpotent case}\label{part3}

Let $G$ be a 3-step nilpotent simply connected Lie group and $\g$ its Lie algebra. The nilpotency condition means that $\g^4=0$. Taking complement subspaces, we write $\g=\mathfrak{v}_1\oplus\mathfrak{v}_2\oplus\mathfrak{v}_3$, where
\[\mathfrak{v}_3=\g^3, \quad \mathfrak{v}_2\oplus\mathfrak{v}_3=\g^2, \quad \text{and}\;\g=\mathfrak{v}_1\oplus \g^2.\]
This ``linear" grading can fail to be an algebra grading: we only have $[\mathfrak{v}_i,\mathfrak{v}_j]\subset\bigoplus_{k\ge i+j}\mathfrak{v}_k$; in this case, this says in particular that $[\mathfrak{v}_1,\mathfrak{v}_1]\subset\mathfrak{v}_2\oplus\mathfrak{v}_3$ and this is the only possible obstruction to be an algebra grading. More precisely, we have an algebra grading if and only if $[\mathfrak{v}_1,\mathfrak{v}_1]\subset\mathfrak{v}_2$, or equivalently if the projection $[\mathfrak{v}_1,\mathfrak{v}_1]_3$ of $[\mathfrak{v}_1,\mathfrak{v}_1]$ on $\mathfrak{v}_3$ modulo $\mathfrak{v}_2$ is zero.

From this choice of linear grading, we can define a ``corrected" bracket $[\cdot,\cdot]'$, namely defining, for $x_i\in\mathfrak{v}_i$, $x_j\in\mathfrak{v}_j$, $[x_i,x_j]'$ as the projection of $[x_i,x_j]$ to $\mathfrak{v}_{i+j}$ (modulo $\g^{i+j+1}$). This defines a Lie algebra law on $\g$, called the Carnot-graded associated Lie algebra $\textnormal{Car}(\g)$. Let us emphasize that while the isomorphism type of $\textnormal{Car}(\g)$ only depends on $\g$, the homeomorphism $\g\to\textnormal{Car}(\g)$, given here as the identity map, is sensitive to choice of linear grading. For instance, $\g$ might admit a Carnot grading but the chosen linear grading is not one. 

We can view $G$ as $\g$ endowed with the Baker-Campbell-Hausdorff (BCH) law:
\[x\ast y=x+y+\frac12[x,y]+\frac1{12}([x,[x,y]]+[y,[y,x]]),\]
which is truncated by the 3-nilpotency condition. Then the BCH law, using the graded bracket $[\cdot,\cdot]'$, defines another group law $\ast'$, and the identity map thus defines a homeomorphism between the two Lie groups $G=(\g,\ast)$ and $\textnormal{Car}(G)=(\g,\ast')$. If the linear grading is a grading, this is an isomorphism, and hence a quasi-isometry. But otherwise it is not likely to be a quasi-isometry; yet it is an SBE. This follows from computing $x\ast y-x\ast' y$, where $x=x_1+x_2+x_3$, $y=y_1+y_2+y_3$ are the decompositions in the linear grading and $[u_1,v_1]_k$, for $u_1,v_1\in\mathfrak{v}_1$, is the projection of the bracket in $\mathfrak{v}_k$:
\[x\ast y-x\ast' y=\frac12[x_1,y_1]_3.\]
To check the metric properties, we fix a norm on $\g$ such that the bracket $[\cdot,\cdot]_3$ is submultiplicative and use the Guivarch norm:
\[\lfloor u\rfloor=\|u_1\|+\|u_2\|^{1/2}+\|u_3\|^{1/3}.\]
Then from we obtain the estimate
\[\lfloor x\ast y-x\ast' y\rfloor= \left\lfloor\!\frac{1}{2}[x_1,y_1]_3\!\right\rfloor= \left\|\frac12[x_1,y_1]_3\right\|^{1/3}\!\!\!\le (\lfloor x_1\rfloor\lfloor y_1\rfloor)^{\frac23}\le \max(\lfloor x\rfloor,\lfloor y\rfloor)^{\frac23}.\]

Then for any distances $d,d'$ induced by left-invariant Riemannian metrics on $G$ and $\textnormal{Car}(G)$, this is a good approximation of the length, in the sense that there exist constants $C\ge 1$, $C'\ge 0$ such that for all $u$
\[d(1,u),\,d'(1,u)\in \big[C^{-1}\lfloor u\rfloor-C',C\lfloor u\rfloor+C'\big],\]
abridged as $\approxeq\lfloor u\rfloor$ Hence
\[d'(x,y)=d'(1,(-x)\ast'y)\approxeq\lfloor (-x)\ast'y\rfloor,\]
and, by subadditivity of the Guivarch norm, 
\[\lfloor (-x)\ast'y\rfloor\le \lfloor (-x)\ast y\rfloor +\lfloor (-x)\ast'y-(-x)\ast y\rfloor\le \lfloor (-x)\ast y\rfloor+\max(\lfloor x\rfloor,\lfloor y\rfloor)^{2/3}.\]
Therefore we obtain
\[d'(x,y)\le C^2d(x,y)+O\big((d(1,x)+d(1,y))^{2/3}\big),\]
as well as the same inequality with $d$ and $d'$ switched.
This means that the identity map $(\g,\ast)\to (\g,\ast')$ is an $O(r^{2/3})$-SBE.

\subsection{On the 4-step nilpotent case}\label{part4}

We can work in a similar fashion in general; in the 4-step nilpotent case, let us also compute $x\ast'y-x\ast y$; note that it only involves terms of degree $\le 3$ in the BCH formula, because terms of degree 4 will cancel, for the same reason as terms of degree 3 canceled in the previous computation.

\begin{align*}
x\ast'y-x\ast y=&\frac12[x_1,y_1]_3+\frac12[x_1,y_1]_4+\frac12([x_1,y_2]_4+[x_2,y_1]_4)\\ +& \frac1{12}([x_1,[x_1,y_1]_3]+[y_1,[y_1,y_1]_3]+[x_1,[x_1,y_1]_2]_4+[y_1,[y_1,x_1]_2]_4).\end{align*}

This is a sum of four terms (as gathered in parentheses), say $M_1+M_2+M_3+M_4$, which metrically behave differently. For $r=\max(\lfloor x\rfloor, \lfloor y\rfloor)$, we have 
\[\lfloor M_1\rfloor\le r^{2/3},\quad \lfloor M_2\rfloor\le r^{1/2},\lfloor M_3\rfloor\le r^{3/4},\quad\lfloor M_4\rfloor\le r^{3/4}.\]

Therefore the same argument shows that the identity map is an $O(r^{3/4})$-SBE and it is not hard to check in general, in the $c$-step nilpotent case, that we obtain an $O(r^{1-c^{-1}})$-SBE. However, this can be improved, because it can happen that some of the above terms vanish. For instance, if $M_2=M_3=M_4=0$; in this case we obtain an $O(r^{1/2})$-SBE. This occurs if and only if the linear grading satisfies $[\mathfrak{v}_1,\mathfrak{v}_1]\subset\mathfrak{v}_2\oplus\mathfrak{v}_4$ and $[\mathfrak{v}_1,\mathfrak{v}_2]\subset\mathfrak{v}_3$. 

This motivates to develop some set-up, so as to express properties asserting that a linear grading satisfies some partial algebra grading conditions. Then one has to wonder whether a nilpotent Lie algebra admits a grading with such conditions. Actually, to avoid idle formalism, we need to introduce conditions that are computably checkable, which will be expressed in terms of higher derivations and will then be reflected as grading conditions. 

We will come back to the 4-step nilpotent case to illustrate the general definitions.

\subsection{Weak derivation conditions}\label{wdc}

Let us begin by some motivation. A Lie algebra, say finite-dimensional over a field of characteristic zero, is Carnot if it admits a grading such that the Lie algebra is generated by elements of degree 1. It is convenient to observe that this is equivalent to be a nilpotent Lie algebra possessing a derivation inducing the identity map modulo the derived subalgebra. A useful observation is that the latter condition can be viewed as the existence of a solution to some (affine) system of linear equations. We are going to introduce similar weaker notions, which will prove relevant to the study of SBEs between nilpotent groups.

\subsubsection{Compatible linear gradings and grading operators}\label{graop}
Let $\g$ be a Lie algebra over a field $K$. 
Let $(\g^i)$ be the lower central series of $\g$. 

A {\it compatible linear grading} of $\g$ is a linear decomposition $\g=\bigoplus\mathfrak{v}_i$ such that $\g^i=\bigoplus_{j\ge i}\mathfrak{v}_j$ for all $i$.

Let $\mathcal{L}^\triangledown(\g)$ be the subalgebra of $\mathcal{L}(\g)$ of those linear endomorphisms stabilizing $\g^i$ for all $i$; we call its elements {\it pregrading operators} of $\g$. We define $\mathcal{D}(\g)$ as the affine subspace of those $D\in\mathcal{L}^\triangledown(\g)$ inducing multiplication by $i$ on $\g^i/\g^{i+1}$ for all $i$. We call elements of $\mathcal{D}(\g)$ {\it grading operators} of $\g$.

There is a canonical bijection between the set of compatible linear gradings and the space of grading operators, written precisely in the following proposition, whose proof is immediate:
\begin{prop}
Assume that $\g$ is nilpotent and that the ground field has characteristic zero (or some $p$ greater than the nilpotency length).
Then there is a canonical bijection between the set $\mathcal{D}(\g)$ of grading operators and the set of compatible linear gradings. Namely, every compatible linear grading comes from a grading operator $D$, defined to be multiplication by $i$ on $\g^i$. Conversely, for a grading operator $D\in\mathcal{D}(\g)$, define $\mathfrak{v}_i=\mathfrak{v}_i(D)$ as the kernel of $D-i$. Then $\g=\bigoplus\mathfrak{v}_i$ is a compatible linear grading. \qed
\end{prop}

The interest of encoding compatible linear gradings as elements of $\mathcal{D}(\g)$ is that the latter is an affine space.

\begin{rem}
A grading operator can also be encoded in the one-parameter group of linear automorphisms $(\delta_t)_{t\in K^*}$ it generates; here $\delta_t$ acts by multiplication by $t^i$ on $\mathfrak{v}_i(D)$. The usefulness of such subgroups (called one-parameter groups of dilations) was made clear by Breuillard's interpretation of Guivarch's and Pansu's results \cite{Bre} (as well as Goodman's, although Breuillard was not aware of \cite{Goo}). Here we prefer the equivalent data of grading operators, because we will take advantage of the structure of affine space of the set of grading operators. 
\end{rem}

\subsubsection{$\Delta_n$ and $n$-derivations}
Denote by $[x_1,\dots,x_n]$ the (left) iterated bracket $[x_1,[x_2,\dots,[x_{n-1},x_n]\cdots]]$. 

Let $\mathcal{L}^n(\g)$ be the space of $K$-multilinear maps $\g^n\to\g$ (for $n=1$, this is the algebra of linear endomorphisms of $\g$ and we simply denote it by $\mathcal{L}(\g)$). For $n\ge 2$, define a linear operator $\Delta_n$ from $\mathcal{L}(\g)$ to $\mathcal{L}^n(\g)$ by
\[(\Delta_n D)(x_1,\dots,x_n)=D[x_1,\dots,x_n]-\]\[([Dx_1,x_2,\dots,x_n]+[x_1,Dx_2,x_3,\dots,x_n]+\dots +[x_1,\dots,x_{n-1},Dx_n]).\]

Note that $\Delta_nD$ is alternating in the last two variables $(x_{n-1},x_n)$.
By definition, $D$ is a derivation if $\Delta_2D=0$. As a generalization, elements of the kernel of $\Delta_n$ are called $n$-{\it derivations} (as introduced in \cite{Abd}, although the notion of 3-derivation, or triple derivation, occurred much earlier in \cite{Lis}). Note that every derivation is an $n$-derivation for all $n\ge 2$. In the sequel, it will be convenient to deal with $\Delta_n$ and not only with its kernel.

\subsubsection{Weak Carnot conditions}
We generalize the notion of being an $n$-derivation in two ways. Being an $n$-derivation for $D\in\mathcal{L}^\triangledown(\g)$ can be rewritten as $(\Delta_nD)(\g,\dots,\g)=0$. First, we allow this to hold modulo some term of the lower central series; second, we restrict the variables to belong to some given terms of the lower central series.

Namely, fix $n\ge 2$ and an $n$-tuple of integers $\wp=(\wp_1,\dots,\wp_n)$ with each $\wp_k\ge 1$ and $j\ge 1$; write $|\wp|=\sum_k\wp_k$. Then define $\mathcal{L}_{(\wp_1,\dots,\wp_n)}^j(\g)$ as the linear subspace of $\mathcal{L}(\g)$ consisting of those pregrading operators $D\in\mathcal{L}^\triangledown(\g)$ such that 
\[\Delta_nD(\g^{\wp_1},\dots,\g^{\wp_n})\subset \g^{j+1}.\]

In particular, for $\wp=(1,\dots,1)$ ($n$ times), this is the space of elements of $\mathcal{L}^\triangledown(\g)$ inducing an $n$-derivation of $\g/\g^{j+1}$. Also it is immediate that $\mathcal{L}_{(\wp_1,\dots,\wp_n)}^j(\g)=\mathcal{L}^\triangledown(\g)$ as soon as $j\le|\wp|$. Also, if $\g$ is $c$-step nilpotent, $\mathcal{L}_{(\wp_1,\dots,\wp_n)}^j(\g)=\mathcal{L}_{(\wp_1,\dots,\wp_n)}^c(\g)$ for all $j\ge c$. Accordingly, it is no restriction to consider only those $(\wp_1,\dots,\wp_n)$ and $j$ when 
\[3\le j\le c,\;\;2\le n<j,\;\;\wp_k\ge 1,\;\;|\wp|<j,\;\;\wp_{n-1}\le \wp_n.\]
For given $j$, write $\mathcal{T}_j$ as the set of such $n$-tuples. Thus
\[\mathcal{T}_2=\emptyset;\;\mathcal{T}_3=\{(1,1)\};\;\mathcal{T}_4=\{(1,1),(1,2),(1,1,1)\},\]
\[\mathcal{T}_5=\mathcal{T}_4\cup\{(1,3),(2,2),(1,1,2),(2,1,1),(1,1,1,1)\}\dots\]
Let $\mathcal{S}_c$ denote the set of pairs $(\wp|j)$ where $j$ ranges over $\{3,\dots,c\}$ and $\wp$ ranges over $\mathcal{T}_j$, and $\mathcal{S}$ the union of all $\mathcal{S}_c$. Here the sign $|$ just replaces a comma for the sake of readability: an element of $\mathcal{S}$ will be denoted as $(\wp_1,\dots,\wp_n|j)$ rather than $((\wp_1,\dots,\wp_n),j)$.

\begin{defn}
Let $\g$ be a Lie algebra, and denote $\mathcal{D}_\wp^j(\g)=\mathcal{D}(\g)\cap\mathcal{L}_\wp^j(\g)$. Given any subset $A\subset\mathcal{S}$, we say that $\g$ is $A$-derivable if \[\mathcal{D}^A(\g):=\bigcap_{(\wp|j)\in A}\mathcal{D}^j_\wp(\g)\neq\emptyset;\]
we call such $D$ an $A$-derivation.
\end{defn}

Note that we assumed nothing about $\g$ or the ground field, although this definition is (so far) only motivated by the study of nilpotent Lie algebras. Actually, if we define $\g^\infty=\bigcap_n\g^n$, then it follows from the definition that $\g$ is $A$-derivable if and only if $\g/\g^\infty$ is $A$-derivable.

Let us clarify the intuition that it is enough to consider $A\subset\mathcal{S}_c$ for $c$-step nilpotent Lie algebras. Namely, given any $A\subset\mathcal{S}$, define $A_c$ as the ``projection" of $A$ on $\mathcal{S}_c$, that is, the set of $(i|\min(j,c))\in\mathcal{S}_c$ when $(i|j)$ ranges over $A$ (if $c\le 2$, just define $A_c=\emptyset$). Then it is immediate that a $c$-step nilpotent Lie algebra is $A$-derivable if and only if it is $A_c$-derivable. In particular, for given $c$, this gives only finitely many definitions.

Let us emphasize two particular cases:
\begin{itemize}
\item Every Lie algebra $\g$ is $\emptyset$-derivable.
\item A $c$-step nilpotent Lie algebra $\g$ over a field of characteristic zero is $\mathcal{S}$-derivable if and only if it is $\{(1,1|c)\}$-derivable, if and only if it is Carnot. Indeed, Carnot is then easily seen to be equivalent to the existence of a derivation inducing the identity on the abelianization \cite[Lemma 3.10]{gralie}.
\end{itemize}

Thus, for a nilpotent Lie algebra, the various conditions of being $A$-derivable are weakenings of being Carnot. 

$\mathcal{S}_3$ is reduced to $\{(1,1|3)\}$; the first new notions appear for $c=4$, where
\[\mathcal{S}_4=\{(1,1|3),(1,1|4),(1,2|4),(1,1,1|4).\}\]
This yields in principle $2^4$ notions, but wiping obvious redundancies 
simplifies the picture. If $A$ contains $(1,1|4)$, $A$-derivable just means Carnot (and other elements of $A$ become redundant). Also, one easily shows the inclusion $\mathcal{D}_{(1,1,1)}^4(\g)\subset\mathcal{D}_{(1,1)}^3(\g)\cap\mathcal{D}_{(1,2)}^4(\g)$. In particular, when $A$ contains both $(1,1|3)$ and $(1,2|4)$, then $(1,1,1|4)$ is redundant.

\begin{rem}\label{113124}
A naive expectation would be that $\g$ is $A$-derivable if and only if it is $\{\alpha\}$-derivable for every $\alpha\in A$. If it were true, as small-dimensional example seem to illustrate, this would reduce to the study of $A$-derivability when $A$ is a singleton. However, this is not true: there is an 11-dimensional 4-step nilpotent Lie complex algebra (defined over the rationals) that is both $\{(1,1|3)\}$-derivable and $\{(1,2|4)\}$-derivable, but not $\{(1,1|3),(1,2|4)\}$-derivable; see \S\ref{contrex11}. This shows that beyond the Carnot case, there is, in general, no ``best" linear compatible grading.
\end{rem}

\begin{prop}\label{ext_scal}
For finite-dimensional Lie algebras $\g$ and $A\subset\mathcal{S}$, the property of being $A$-derivable is invariant under extensions of scalars. Namely, for any extension field $L$ of $K$, write $\g_L=L\otimes_K\g$; viewed it as a Lie algebra over $L$. Then $\g$ is $A$-derivable if and only if $\g_L$ is $A$-derivable.
\end{prop}
\begin{proof}
The condition can be written as a system of (affine) linear equations with coefficients in $K$; in particular it has a solution in $L$ if and only if it has a solution in $K$.
\end{proof}

\begin{rem}
That being $A$-derivable is given by a system of linear equations also means that it can be checked computationally (if we can input the Lie algebra constants).
\end{rem}

\subsubsection{Interpretation in terms of linear gradings}\label{inteli}
Let $\g$ be a nilpotent Lie algebra over a field of characteristic zero.
Any $D\in\mathcal{D}(\g)$ defines a grading $\g=\bigoplus\mathfrak{v}_i$ as in \S\ref{graop}. For $\wp=(\wp_1,\dots,\wp_n)$, write $|\wp|=\sum \wp_k$. Then for all $x_k\in\mathfrak{v}_{\wp_k}$, we have 
\[\Delta_nD(x_1,\dots,x_n)=D([x_1,\dots,x_n])-|\wp|[x_1,\dots,x_n].\]

Decompose the bracket according to this grading: for each given $\wp$ and all $x_k\in\mathfrak{v}_{\wp_k}$, write 
\[[x_1,\dots,x_n]=\sum_{j\ge |\wp|}[x_1,\dots,x_n]_j.\]
Then 
\[\Delta_nD(x_1,\dots,x_n)=\sum_{j>|\wp|}(j-|\wp|)[x_1,\dots,x_n]_j.\]

For $\g$ $c$-step nilpotent and $\wp=(\wp_1,\dots,\wp_n)$ the condition $D\in\mathcal{D}^c_\wp(\g)$ means that $\Delta_nD(\mathfrak{v}_{\wp'_1},\dots,\mathfrak{v}_{\wp'_n})=0$ for all $\wp'\ge\wp$ (that is, all $\wp'_1\ge \wp_1\dots \wp'_n\ge \wp_n$), which thus means that $[\mathfrak{v}_{\wp'_1},\dots,\mathfrak{v}_{\wp'_n}]_j=0$ for all $\wp'\ge \wp$ with $|\wp'|<c$ and all $j$ such that $|\wp'|<j\le c$.

In general, this means that this condition holds modulo $\g^{c+1}$. Let us write, for record
\begin{prop}\label{caractdji}
Assume that the field has characteristic zero and $\g$ is nilpotent. For any $\wp,j$, and any $D\in\mathcal{D}(\g)$ defining a linear grading $(\mathfrak{v}_i)$, the condition $D\in\mathcal{D}^j_\wp(\g)$ means that 
$[\mathfrak{v}_{\wp'_1},\dots,\mathfrak{v}_{\wp'_n}]_\ell=0$ for all $\wp'\ge\wp$ and all $\ell$ such that $|\wp'|<\ell\le j$.\qed
\end{prop}

For instance, for $c=4$, $i\in\{(1,1),(1,2),(1,1,1)\}$. Then 
\begin{itemize}
\item $D\in\mathcal{D}^4_{(1,1)}(\g)$ means that $D$ is a Carnot grading (or equivalently satisfies $[\mathfrak{v}_1,\mathfrak{v}_1]\subset\mathfrak{v}_2$, $[\mathfrak{v}_1,\mathfrak{v}_2]\subset\mathfrak{v}_3$; 
\item $D\in\mathcal{D}^4_{(1,2)}(\g)$ means that $[\mathfrak{v}_1,\mathfrak{v}_2]\subset\mathfrak{v}_3$;
\item $D\in\mathcal{D}^4_{(1,1,1)}(\g)$ means that that $[\mathfrak{v}_1,[\mathfrak{v}_1,\mathfrak{v}_1]]\subset\mathfrak{v}_3$.
\item In addition, $D\in\mathcal{D}^3_{(1,1)}(\g)$ means that $[\mathfrak{v}_1,\mathfrak{v}_1]\subset\mathfrak{v}_2\oplus\mathfrak{v}_4$.
\end{itemize}

Thus, being an $A$-derivation encodes some of the conditions involved in being a Lie algebra grading. 

\begin{rem}Not all possible partial conditions are encoded in this way; for instance, for $c=4$ the condition $[\mathfrak{v}_1,\mathfrak{v}_1]\subset\mathfrak{v}_2$ does not appear in such a way, so Proposition \ref{ext_scal} does not apply to the existence of such a grading.

More precisely, in a 4-step nilpotent Lie algebra $\g$, the set of grading operators for which the corresponding grading satisfies $[\mathfrak{v}_1,\mathfrak{v}_1]\subset\mathfrak{v}_2$ is not always affine subspace of $\mathcal{D}(\g)$. For instance, let $\g$ be the standard filiform 5-dimensional Lie algebra, with basis $(e_i)_{1\le i\le 5}$ and nonzero brackets $[e_1,e_i]=e_{i+1}$, $i=2,3,4$. Consider the 1-dimensional affine subspace $V$ of $\mathcal{D}(\g)$ consisting of those operators $D_x$ for $x$ in the ground field, where $d_x$ is defined by $e_1\mapsto e_1$, $e_2\mapsto e_2+xe_3+e_4$, $e_3\mapsto 2e_3+xe_4$, $e_4\mapsto 3e_4$, $e_5\mapsto 4e_5$. Then if $(\mathfrak{v}_i)$ is the corresponding grading (for a given $x$), a basis of $\mathfrak{v}_1$ is $(e_1,e'_2(x))$, where $e'_2(x)=e_2-xe_3+\frac12(x^2-1)e_4$. Then computation yields \[D_x[e_1,e'_2(x)]-2[e_1,e'_x]=(x^2-1)e_5.\]
So $[\mathfrak{v}_1,\mathfrak{v}_1]\subset\mathfrak{v}_2$ if and only if $x^2=1$. This means that the set of grading operators whose grading satisfies $[\mathfrak{v}_1,\mathfrak{v}_1]\subset\mathfrak{v}_2$ does not intersect $V$ in an affine subspace, so is not an affine subspace.
\end{rem}

\subsubsection{The $e$-invariant}

We now focus on certain particular subsets $A$ of $\mathcal{S}$. Namely, for $r\in\R$, define $[r]=\{(\wp|j):|\wp|/j>r\}$. Here $\wp$ ranges over all $n$-tuples $(\wp_1,\dots,\wp_n)$ with all $\wp_k\ge 1$ such that $|\wp|<j$ and $n$ ranges over $\{2,\dots,j-1\}$. (Recall that $|\wp|$ means $\sum_k\wp_k$, and $(\wp|j)$ is just one way to write $(\wp,j)$). 

Since $\wp'\ge \wp$ implies $|\wp'|/j\ge |\wp|/j$, Proposition \ref{caractdji} implies the following.

\begin{prop}Assume that the field has characteristic zero and $\g$ is nilpotent.
A grading operator $D\in\mathcal{D}(\g)$ is an $[r]$-derivation if and only if, considering the corresponding linear grading $(\mathfrak{v}_i)$, for all $\wp$ and $j$ such that $|\wp|/j>r$, we have
\[[\mathfrak{v}_{\wp_1},\dots,\mathfrak{v}_{\wp_n}]\subset \mathfrak{v}_{|\wp|}\oplus\g^{j+1}.\]
\end{prop}

\begin{defn}\label{defe}
For a nilpotent Lie algebra $\g$ over a field of characteristic zero, define
$e_\g=\inf_{D\in\mathcal{D}(\g)}e_D$, where, for a given $D\in\mathcal{D}(\g)$,
\[e_D=\inf\{r\ge 0:\;D\textnormal{ is a }[r]\textnormal{-derivation}\}.\]
Thus $e_\g$ is also the infimum of $r\ge 0$ such that $\g$ is $[r]$-derivable.
\end{defn}

Note that $e_D=0$ means that $D$ is a derivation, so $e_\g=0$ means that $\g$ is Carnot.

Define $\mathcal{J}_c=\{0\}\cup\{i/j:2\le i<j\le c\}$. So
\[\mathcal{J}_2=\{0\},\;\mathcal{J}_3=\left\{0,\frac23\right\},\;\mathcal{J}_4=\left\{0,\frac12,\frac23,\frac34\right\}, \mathcal{J}_5=\left\{0,\frac25, \frac12,\frac35,\frac23,\frac34,\frac45\right\},\dots\]
Suppose that $\g$ is $c$-nilpotent. For $r\in [0,1]$, let $s$ be the largest number $\le r$ in $\mathcal{J}_c$. Then $D$ is an $[r]$-derivation if and only if $\g$ is an $[s]$-derivation. It follows that the infimums in Definition \ref{defe} are attained and belong to $\mathcal{J}_c$. It particular, they are $\le 1-c^{-1}$. We will see in Proposition \ref{alle} that all values of $\mathcal{J}_c$ can be achieved by some finite-dimensional $c$-step nilpotent Lie algebra, which can be chosen to be defined over $\Q$.

\begin{cor}\label{edcro}
Assume that the field has characteristic zero and $\g$ is nilpotent.
For every grading operator $D\in\mathcal{D}(\g)$, We have
\[[\mathfrak{v}_{\wp_1},\dots,\mathfrak{v}_{\wp_n}]\subset \mathfrak{v}_{|\wp|}\oplus\g^{\lceil|\wp|/e_D\rceil}.\qquad\qed\]
\end{cor}

\subsubsection{Example and miscellaneous facts on the $e$-invariant}\label{misce}

For a 3-step nilpotent Lie algebra, we have $e_\g\in\{0,2/3\}$, and it is $0$ if and only if it is Carnot. In general, it is always true that $e_{\g/\g^j}\le e_\g$ for all $j$. Thus, $e_\g<2/3$ implies that $\g/\g^4$ is Carnot.

For a 4-step nilpotent Lie algebra, we more precisely have $e_\g\in\{0,1/2,2/3,3/4\}$ and is determined by the following three facts, stated in terms of gradings and of derivations:
\begin{enumerate}
\item $e_\g=0$ $\Leftrightarrow$ $\g$ is Carnot $\Leftrightarrow$ there exists a derivation in $\mathcal{D}(\g)$;
\item\label{it2} $e_\g\in\{0,1/2\}$ $\Leftrightarrow$ there exists a linear grading such that $[\mathfrak{v}_1,\mathfrak{v}_1]\subset\mathfrak{v}_2\oplus\mathfrak{v}_4$ and $[\mathfrak{v}_1,\mathfrak{v}_2]\subset\mathfrak{v}_3$ $\Leftrightarrow$ there exists $D\in\mathcal{D}(\g)$ such that $\Delta_2D$ maps into $\g^4$ and maps $\g^1\times\g^2$ to zero;
\item\label{it3} $e_\g\in\{0,1/2,2/3\}$ $\Leftrightarrow$ there exists a linear grading such that $[\mathfrak{v}_1,[\mathfrak{v}_1,\mathfrak{v}_1]]\subset\mathfrak{v}_3$ and $[\mathfrak{v}_1,\mathfrak{v}_2]\subset\mathfrak{v}_3$ $\Leftrightarrow$ there exists $D\in\mathcal{D}(\g)$ such that $\Delta_2D$ maps $\g^1\times\g^2$ to zero and $\Delta_3D$ maps $\g^1\times\g^1\times\g^1$ to zero.
\end{enumerate}

Let us provide a comprehensive description for nilpotent Lie algebras of dimension $\le 6$ in characteristic zero, where all the previous cases actually occur. By Proposition \ref{ext_scal}, it is enough to consider Lie algebras over an algebraically closed field, where the classification is simpler.

In the table below, the first line concerns the Carnot case. The second line concerns the case of all non-Carnot 3-step nilpotent Lie algebras $\g$: up to isomorphism and for an algebraically closed field, there is one such Lie algebra in dimension 5, and 5 in dimension 6, including the product of the 5-dimensional one with a 1-dimensional abelian one. Then we list the remaining ones, namely those of nilpotency length $\ge 4$. We mention both the notation in De Graaf \cite{graaf} and Magnin \cite{Mag} ($L_{i,j}$ refers to \cite{graaf} and $\g_{i,j}$ to \cite{Mag}). In each case we write the law, writing only nonzero brackets, and shortening such a notation as $[e_i,e_j]=e_k$ to ``$ij:k$" for a choice of basis $(e_1,\dots)$. The columns $\tau$ gives the sequence of nonzero dimensions $(\g^i/\g^{i+1})$: the number of terms gives the nilpotency length $c$ and the sum of terms gives the dimension $d$; for $\tau$, we abbreviate, for instance, $(2,1,1)$ into 211. The column ``failure" gives a set $A$ (here, always a singleton) for which $\g$ fails to be $A$-derivable and yielding the lower bound for $e_\g$.\smallskip

\begin{tabular}{|l|l|l|l|l|l|}
\hline
     Lie algebra &  $\tau$ & d & c & $e_\g$ & failure \\
\hline\hline
   any Carnot &  & & & $0$ & $-$ \\
\hline
  any 3-step, non-Carnot & & & 3 & 2/3 & $\{11\vert3\}$\\
\hline
 $\g_{5,5}=L_{6,7}$:  12:3, 13:4, 14:5, 23:5 &2111 & 5 & 4 & 3/4  &$\{12\vert 4\}$\\
\hline
 $\g_{6,11}=L_{6,12}$:  12:4, 14:5, 15:6, 23:6 &3111 & 6 & 4 & 1/2&$\{11\vert4\}$\\
\hline
 $\g_{6,12}=L_{6,11}$:  12:4, 14:5, 15:6, 23:6, 24:6 &3111 & 6 & 4 & 3/4&$\{12\vert4\}$\\
\hline
 $\g_{6,13}=L_{6,13}$:  12:4, 14:5, 15:6, 23:5, 43:6 &3111 & 6 & 4 & 3/4&$\{12\vert4\}$\\
\hline
 $\g_{6,17}=L_{6,17}$:  12:3, 13:4, 14:5, 15:6, 23:6 &21111 & 6 & 5 & 3/5&$\{12\vert5\}$\\
\hline
 $\g_{6,19}=L_{6,15}$:  12:3, 13:4, 14:5, 15:6, 23:5, 24:6 &21111 & 6 & 5 & 4/5&$\{13\vert5\}$\\
\hline
 $\g_{6,20}=L_{6,14}$:  12:3, 13:4, 14:5, 25:6, 23:5, 43:6 &21111 & 6 & 5 & 4/5&$\{13\vert5\}$\\
\hline
\end{tabular}

\smallskip
It is easy to see that taking the direct product with an abelian Lie algebra does not affect the notion of $A$-derivability (in either direction), and hence does not the value of $e$. So the product of $\g_{5,5}$ with a 1-dimensional Lie algebra was omitted in this table (it has $e=3/4$).

Let us illustrate the proof (of the value of $e_\g$) in only the case of $\g=\g_{6,11}$, since all others are similar: denoting by $(v_i)_{1\le i\le 6}$ basis, the linear compatible grading for which the basis elements $v_1,v_2,v_3$ have degree 1, $v_4$ has degree 2, $v_5$ has degree 3 and $v_6$ has degree 4, satisfies all algebra grading conditions except for $[\mathfrak{v}_1,\mathfrak{v}_1]$, which is contained in $\mathfrak{v}_2\oplus\mathfrak{v}_4$. Thus it defines a $[r]$-derivation for all $r>1/2$, and hence $e_\g\le 1/2$. But there is no grading defining a $[1/2]$-derivation. Indeed, the corresponding new grading should satisfy $[\mathfrak{v}_1,\mathfrak{v}_1]\subset\mathfrak{v}_2$. But there should be some elements $v_2+w$, $v_3+x$ of degree 1 with $w,x$ in the derived subalgebra $\g^2$. Then, since both $v_2$ and $v_3$ centralize the derived subalgebra which is abelian, we obtain $[v_2+w,v_3+x]=v_6$, a contradiction, showing $e_\g\ge 1/2$. Another explicit case where we will perform such a verification is given in \S\ref{contrex11}.

\begin{rem}
It would have been tempting to restrict all this discussion to the case of $\wp$ of length $2$, that is, ignore $n$-derivations for $n\ge 3$. But although this is not visible in the above table, this would affect the value of $e_\g$, including in the 4-step nilpotent case. Indeed, consider the 7-dimensional Lie algebra denoted by $\g_{7,1,2(i_1)}$ in \cite{Mag}. With the above conventions, its law can be written as
\[\textnormal{12:4, 14:5, 24:6, 15:7, 26:7, 13:6, 23:5}.\]
Then this Lie algebra is $\{(1,2|4)\}$-derivable but not $\{(1,1,1|4)\}$-derivable (in particular $e_\g=3/4$). The more complicated example of \S\ref{contrex11} also satisfies this, but unlike this one, is also $\{(1,1|3)\}$-derivable.
\end{rem}

\begin{prop}\label{alle}
For every $c\ge 3$, all $3\le j\le c$ and all $i\in\{2,\dots,j-1\}$, there exists a finite-dimensional $c$-step nilpotent Lie algebra (defined over $\Q$) such that $e_\g=i/j$.
\end{prop}
\begin{proof}
Consider the Lie algebra with basis $X,Y_1,\dots,Y_{i-1},U,V_1,\dots,V_j$ and nonzero brackets
\[[X,Y_p]=Y_{p+1},\;1\le p\le i-2,\;[X,Y_{i-1}]=V_j,\quad [U,V_q]=V_{q+1},\;1\le q\le j-1.\]
(This is a central product of two standard filiform Lie algebras, of dimension $i+1$ and $j+1$.) 
Consider the compatible linear grading for which $X,U$ have degree 1, $Y_p$ has degree $p$ and $V_q$ has degree $q$. Then it is immediate that it belongs to $\mathcal{D}_\wp^k$ for all $(\wp|k)$ such that either $k<j$ or $|\wp|\neq i$. In particular, $e_\g\le e_D\le i/j$, where $D$ is the corresponding grading operator. 

To prove that $e_\g\ge i/j$, it is enough to show that $\mathcal{D}_{(1,i-1)}^j(\g)=\emptyset$. Indeed, if $E$ belongs to it, then $E(V_j)=jY_j$, and also, defining $N=E-D$, we have
\[E(V_j)=E([X,Y_{i-1}])=[EX,Y_{i-1}]+[X,EY_{i-1}]=iV_j+[NX,Y_{i-1}]+[X,NY_{i-1}].\]
Since $NX$ is a linear combination of the $Y_p$ and $Z_q$, and $NY_{i-1}$ is a linear combination of the $V_q$, we see that both brackets vanish and thus $jV_j=E(V_j)=iV_j$. This is a contradiction and thus $E$ cannot exist. Thus $e_\g=i/j$. 
\end{proof}

\subsection{A counterexample}\label{contrex11}

Here we give the example illustrating Remark \ref{113124}; we write in a separate subsection because of the length of the proof, and because the reader can skip it in a first reading if not specifically interested.

Consider the 11-dimensional Lie algebra with basis
\[(a_1,a_2,b_1,b_2,c_1,c_2,c_3,d_1,d_2,e_1,e_2)\]
 and nonzero products (we omit the brackets)
\[a_1a_2=b_2,\quad a_1b_1=c_1,\;b_1a_2=c_2,\;a_1b_2=c_3,\quad a_1c_1=d_1,\;a_1c_2=b_1b_2=d_2;\]
\[b_1c_1=b_2c_1=a_1d_1=d_1a_2=e_1,\; a_1d_2=e_2,\; b_1c_3=e_1+e_2.\]

This is a priori an algebra with alternating product. We have to show that the Jacobi form, which is alternating trilinear, vanishes on triples of distinct basis elements. 

Note that this algebra has a grading in $\{1,\dots,5\}$ for which $a_i$ has degree 1, $b_i$ degree 2, etc. (This already shows that it is nilpotent.) So to check the Jacobi identity, it is enough to consider triples of distinct basis elements with total weight $\le 5$.

The only 3-element subsets of the basis involving $a_2$ and from which we can form at least one nonzero triple product are $\{a_1,a_2,b_1\}$ and $\{a_1,a_2,c_1\}$. In both cases, the Jacobi form vanishes anyway:

\[[a_1,[a_2,b_1]]+[a_2,[b_1,a_1]]+[b_1,[a_1,a_2]]=\qquad\qquad\qquad\qquad\]
\[\qquad\qquad-[a_1,c_2]-[a_2,c_1]+[b_1,b_2]=-d_2+0+d_2=0;\] 
\[[a_1,[a_2,c_1]]+[a_2,[c_1,a_1]]+[c_1,[a_1,a_2]]=\qquad\qquad\qquad\qquad\] 
\[\qquad\qquad0-[a_2,d_1]+[c_1,b_2]=0+e_1-e_1=0.\]

The only remaining 3-element subset of the basis with total weight $\le 5$ is $\{a_1,b_1,b_2\}$, and we have
\[[a_1,[b_1,b_2]]+[b_1,[b_2,a_1]]+[b_2,[a_1,b_1]]=\qquad\qquad\qquad\qquad\]
\[\qquad\qquad[a_1,d_2]-[b_1,c_3]+[b_2,c_1]=e_2-(e_1+e_2)+e_1=0.\]
Hence this is indeed a Lie algebra. 

\begin{prop}
The above Lie algebra is both $\{(1,1|3)\}$-derivable and $\{(1,2|4)\}$-derivable, but not $\{(1,1|3),(1,2|4)\}$-derivable. Actually, it is not $\{(1,1,1|4)\}$-derivable.
\end{prop}
\begin{proof}
It is straightforward that a compatible linear grading is obtained by taking $a_1,a_2,b_1$ to be of degree 1, $b_2,c_1,c_2$ of degree 2, $c_3,d_1,d_2$ of degree 3, and $e_1,e_2$ of degree 4. In particular, the nilpotency length is 4. 

The only basis relation not preserving this grading is $[b_1,c_1]=e_1$, which has the form ``$1+2=4$". In particular, the quotient by the fourth term of the central series is Carnot, and thus the Lie algebra is $\{(1,1|3)\}$-derivable.

We can now change the basis, by replacing $b_1$ by $b'_1=b_1-b_2$. We obtain another compatible linear grading by replacing $b_1$ with $b'_1$ in the previous description, which we now denote as $\bigoplus_{i=1}^4\mathfrak{v}_i$. Then we see that the latter defines a $(1,2|4)$-derivation, or equivalently that $[\mathfrak{v}_1,\mathfrak{v}_2]\subset\mathfrak{v}_3$. Indeed, in this new basis, the nonzero brackets are
\[a_1a_2=b_2,\quad a_1b'_1=c_1-c_3,\;b'_1a_2=c_2,\;a_1b_2=c_3,\quad a_1c_1=d_1,\;a_1c_2=b'_1b_2=d_2;\]
\[b_2c_1=a_1d_1=d_1a_2=e_1,\; a_1d_2=e_2,\; b'_1c_3=e_1+e_2.\]

We see now that the only basis relation not preserving the new linear grading is $[a_1,b'_1]=c_1-c_3$; in particular, $[\mathfrak{v}_1,\mathfrak{v}_2]\subset\mathfrak{v}_3$. Thus $\g$ is $\{(1,2|3)\}$-derivable.

(Note that if we replace $c_1$ by $c'_1=c_1-c_3$, then we obtain the new basis relation $[b'_1,c'_1]=-e_1-e_2$, and again the resulting grading operator, albeit being again a $\{(1,1|3)\}$-derivation, fails to define a $\{(1,2|3)\}$-derivation.)

Let us actually prove by contradiction that there is no grading $\g=\bigoplus_{i=1}^4\mathfrak{v}_i$ whose associated grading operator is a $\{(1,1|3),(1,2|4)\}$-derivation. This means that $[\mathfrak{v}_1,\mathfrak{v}_1]\subset \mathfrak{v}_2\oplus\mathfrak{v}_4$ and $[\mathfrak{v}_1,\mathfrak{v}_2]\subset \mathfrak{v}_3$. This implies $[\mathfrak{v}_1,[\mathfrak{v}_1,\mathfrak{v}_1]]\subset\mathfrak{v}_3$, i.e., there is a $\{(1,1,1|4)\}$-derivation; we are going to deduce a contradiction.

Under this contradictory assumption, $\mathfrak{v}_1$ contains two elements of the form $a=a_1+w$ and $b=b_1+x$, with both $x,w$ in $[\g,\g]$. Then $[b,[a,b]]=[b,a,b]$ has to be of degree 3. Now we compute:
\[[b,a,b]=[b_1,a_1,b_1]+[b_1,w,b_1]+[x,a_1,b_1]+[b_1,a_1,x]+(\dots),\]
where $(\dots)$ is a sum of brackets involving two terms of the derived subalgebra and an additional term, and thus belong to the fifth term of the lower central series, and hence vanishes. Also for the ordinary Lie algebra grading, $[b_1,w,b_1]$ has to be a sum of terms of degree $\ge 6$ and hence vanishes; for the same Lie algebra grading, we see that if we write $x=\lambda b_2+x'$ with $x'$ a sum of terms of degree $\ge 3$, then the terms involving $x'$ will also vanish. So we have
\[[b,a,b]=[b_1,a_1,b_1]+\lambda([b_2,a_1,b_1]+[b_1,a_1,b_2])\]
\[=e_1+\lambda(e_1+(e_1+e_2)).\]
Since the latter has to be of degree 3 and is also of degree 4, it vanishes, whence $(2\lambda+1)=\lambda=0$, a contradiction.
\end{proof}

\subsection{SBE of nilpotent groups}\label{su_nil}

If $G$ is a simply connected nilpotent Lie group, we write $e_G=e_\g$, where $\g$ is the Lie algebra of $G$, and $e_\g$ is introduced in Definition \ref{defe}. The following theorem, when $e=1-c^{-1}$, is a quantitative version of Pansu's theorem describing the asymptotic cone \cite{Pan}. The contribution here is the formulation of the result in the context of SBEs, as well as the improvement of the exponent.

\begin{thm}\label{pansuq}
Let $G$ be a simply connected nilpotent Lie group and $G_\infty$ the associated Carnot Lie group, and $e=e_G$. Then $G$ is $O(r^e)$-SBE to $G_\infty$.
\end{thm}

This follows from a more precise result. Let $G$ be a simply connected nilpotent Lie group, $\g$ its Lie algebra. Let $D$ be a grading operator on $\g$ (see \S\ref{graop}), and $(\mathfrak{v}_i)$ the corresponding linear grading. For $x\in\mathfrak{v}_i$, $y\in\mathfrak{v}_j$, define $[x,y]^\infty$ as the projection of $[x,y]$ on $\mathfrak{v}_{i+j}$ modulo $\g^{i+j+1}$; extend it by bilinearity to $\g$; this is a Lie bracket. This is the standard way to construct the associated graded Carnot algebra. Let us emphasize (since this is a frequent point of confusion) that although $(\g,[\cdot,\cdot]^\infty)$, up to graded Lie algebra isomorphism, does not depend on the choice of $D$, the linear isomorphism given by the identity $(\g,[\cdot,\cdot])\to (\g,[\cdot,\cdot]^\infty)$ is very sensitive to it. At the level of Lie groups, it induces a homeomorphism $\Phi_D$ from $G$ onto $G_\infty$.
Since $e_\g=\min_De_D$, the following theorem entails Theorem \ref{pansuq}.

\begin{thm}\label{resbe}
Let $G$ be a simply connected nilpotent Lie group and $G_\infty$ the associated Carnot Lie group. For every grading operator $D$, the homeomorphism $\Phi_D: G\to G_\infty$ is an $O(r^{e_D})$-SBE.
\end{thm}

Define a binary law on $\g$ by $x\ast y=\log(\exp(x)\exp(y))$: this is the group law, transported to $\g$ through the exponential map. If instead we perform this using the $\g$ endowed with the bracket $[\cdot,\cdot]^\infty$, we define another group law $\ast_\infty$ on $\g$ (depending on the original bracket and on $D$).

Fix a norm on each $\mathfrak{v}_i$ and define, for $x=\sum x_i\in\g$ (with $x_i\in\mathfrak{v}_i$), the Guivarch norm $\lfloor x\rfloor=\max\|x_i\|^{1/i}$. 
The main lemma underlying Theorem \ref{resbe} consists in proving the following:

\begin{lem}[Goodman-type inequality]\label{lgoodmane}
Under the assumptions of Theorem \ref{resbe}, there exists a positive constant $C$ such that for all $z_1,z_2$ in $\g$, we have
\[\lfloor(z_1\ast z_2)-(z_1\ast_\infty z_2)\rfloor\le
C\max(1,\lfloor z_1\rfloor^{e_D},\lfloor z_2\rfloor^{e_D}).\]
\end{lem}

Goodman \cite[Theorem 1]{Goo} established this lemma with $e_D$ replaced with $1-c^{-1}$, where $c$ is the nilpotency length. He just stated it with a constant $\gamma<1$ instead, but in his proof the definition of $\gamma$ makes it clear that it is $\le 1-c^{-1}$.

\begin{proof}[Proof of Lemma \ref{lgoodmane}]
Write $e=e_D$.  Note that $\lfloor\cdot\rfloor$ is subadditive.
For any $n$ and $\wp=(\wp_1,\dots,\wp_n)$, we have, for all $x_i\in\mathfrak{v}_{\wp_i}$, using the notation of \S\ref{inteli}, 
\[[x_1,\dots,x_n]-[x_1,\dots,x_n]^\infty=\sum_{j>|\wp|}[x_1,\dots,x_n]_j\]
By Corollary \ref{edcro}, this can be rewritten as
\[[x_1,\dots,x_n]-[x_1,\dots,x_n]^\infty=\sum_{j\ge|\wp|/e}[x_1,\dots,x_n]_j\]
(for the trivial case $e=0$, the latter sum is over $j\ge +\infty$, thus is an empty sum, thus is zero).

Then  for all $x_i\in\mathfrak{v}_{\wp_i}$,
\[\lfloor[x_1,\dots,x_n]-[x_1,\dots,x_n]^\infty\rfloor\le\max_{j\ge|\wp|/e}\lfloor[x_1,\dots,x_n]_j\rfloor=\max_{j\ge|\wp|/e}\|[x_1,\dots,x_n]_j\|^{1/j}.\]
There exists a constant $C_{(\wp,j)}$ (depending only on $\g$ and the choices of norms) such that $\|[x_1,\dots,x_n]_j\|\le C_{(\wp,j)}\prod_{i=1}^n\|x_i\|$ for all $x_i\in\wp_i$. Then, denoting $C=\max_{|\wp|<j\le c}C_{(\wp,j)}^{1/j}$ (where $c$ is the nilpotency length of $\g$), we have, for all $x_i\in\mathfrak{v}_{\wp_i}$
\[
\lfloor[x_1,\dots,x_n]-[x_1,\dots,x_n]^\infty\rfloor\le
C\left(\prod_{i=1}^n\|x_i\|\right)^{1/j}\]
\[=C\left(\prod_{i=1}^n\lfloor x_i\rfloor^{\wp_i}\right)^{1/j}
\le C\left(\max_{i=1}^n\lfloor x_i\rfloor\right)^{|\wp|/j}
\]
Thus for all $x_i\in\mathfrak{v}_{\wp_i}$, we have 
\begin{equation}\label{diffe}
\lfloor[x_1,\dots,x_n]-[x_1,\dots,x_n]^\infty\rfloor\le
C\max\left(1,\left(\max_{i=1}^n\lfloor x_i\rfloor\right)^{e}\right)
\end{equation}

Now write the Baker-Campbell-Hausdorff series as
\[B(z_1,z_2)=z_1+z_2+\sum_{n\ge 2}\sum_{q\in\{0,1\}^n}b_{n,q}[z_{q_1},\dots,z_{q_n}].\]

The Baker-Campbell-Hausdorff formula says, in particular, that in a nilpotent Lie algebra, $B(z_1,z_2)=\log(\exp(z_1)\exp(z_2))$.

Here we have to beware that we consider two Lie algebras brackets, and therefore also have two group structures. Identifying the simply connected Lie group to its Lie algebra through the exponential, we obtain the two group structures $\ast$ and $\ast_\infty$ given by the Baker-Campbell-Hausdorff series computed on the one hand with $[\cdot,\cdot]$, and on the other hand with $[\cdot,\cdot]^\infty$. Thus 
\[(z_1\ast z_2)-(z_1\ast_\infty z_2)=\sum_{n=2}^c\sum_{q\in\{0,1\}^n}b_{n,q}([z_{q_1},\dots,z_{q_n}]-[z_{q_1},\dots,z_{q_n}]^\infty).\]
We now need to write $z_k=\sum_{i=1}^c z_i^k$ in the decomposition $\g=\bigoplus\mathfrak{v}_i$. Then for any multilinear function $f$ on $n$ variables, we have 
\[f(z_{q_1},\dots,z_{q_n})=\sum_{p\in\{1,\dots,c\}^{\{1,\dots,n\}}}f(z^{q_1}_{p_1},\dots,z^{q_n}_{p_n})\]
Hence, writing $\{1,\dots,c\}^{\{1,\dots,n\}}=[c]^{[n]}$ we have
\[(z_1\ast z_2)-(z_1\ast_\infty z_2)=\sum_{n=2}^c\sum_{q\in\{0,1\}^n}
\sum_{p\in[c]^{[n]}}
b_{n,q}([z^{q_1}_{p_1},\dots,z^{q_n}_{p_n}]-[z^{q_1}_{p_1},\dots,z^{q_n}_{p_n}]^\infty),\]
whence
\[\lfloor(z_1\ast z_2)-(z_1\ast_\infty z_2)\rfloor\le \sum_{n=2}^c\sum_{q\in\{0,1\}^n}
\sum_{p\in[c]^{[n]}}
|b_{n,q}|\lfloor[z^{q_1}_{p_1},\dots,z^{q_n}_{p_n}]-[z^{q_1}_{p_1},\dots,z^{q_n}_{p_n}]^\infty\rfloor.\]
By (\ref{diffe}), we deduce 
\[\lfloor(z_1\ast z_2)-(z_1\ast_\infty z_2)\rfloor\le
C\sum_{n=2}^c\sum_{q\in\{0,1\}^n}\sum_{p\in[c]^{[n]}}
|b_{n,q}|\max(1,\max_{i=1}^n\lfloor z_{p_i}^{q_i}\rfloor^e).
\]
Then, denoting $C'=C\sum_{n=2}^c\sum_{q\in\{0,1\}^n}\sum_{p\in[c]^{[n]}}
|b_{n,q}|$, since $\lfloor z_q\rfloor=\max_i\lfloor z^q_i\rfloor$, we have
\[\lfloor(z_1\ast z_2)-(z_1\ast_\infty z_2)\rfloor\le
C'\max(1,\lfloor z_1\rfloor^e,\lfloor z_2\rfloor^e).\qedhere\]
\end{proof}

\begin{proof}[Proof of Theorem \ref{resbe}]
The inverse law of both group laws is given by $z^{-1}=-z$. Denote by $d$ and $d_\infty$ proper left-invariant geodesic distances on $G$ and $G_\infty$. We both view them as distances on $\g$. Then, by Guivarch's estimates \cite{Guiv}, there exists a constant $M\ge 1$ such that for all $z\in\g$, both $d(0,z)$ and $d_\infty(0,z)$ belong to $[M^{-1}\lfloor z\rfloor,M\lfloor z\rfloor]$. 
Then, writing $-z_1\ast z_2$ for $(-z_1)\ast z_2$ and using Lemma \ref{lgoodmane},
\begin{align*}
d(z_1,z_2)= &d(-z_1\ast z_2,0)\le M\lfloor -z_1\ast z_2\rfloor\\
\le & M\lfloor -z_1\ast_\infty z_2\rfloor+ M\lfloor (-z_1\ast z_2)-(-z_1\ast_\infty z_2)\rfloor\\
\le & M^2d_\infty(z_1,z_2)+MC\max(1,\lfloor z_1\rfloor^{e_D},\lfloor z_2\rfloor^{e_D}),
\end{align*}
and exactly the same reasoning works exchanging $d$ and $d_\infty$. Therefore the identity map $(\g,d)\to (\g,d_\infty)$ is $O(r^{e_D})$-Lipschitz as well as its inverse. Therefore it is an $O(r^{e_D})$-SBE.
\end{proof}

Let us elaborate on Question \ref{q_sbenil}. Every grading operator $D$ on $\g$ yields an SBE $\Phi_D:G\to G_\infty$, well-defined up to composition by an automorphism of $G_\infty$. The question can be split into two parts.

\begin{que}\label{sbeni_c}
Given an admissible function $r\mapsto f(r)\ge 1$,
\begin{enumerate}
\item\label{sbeni_c1} is it true that 
$r^e=O(f(r))$ if and only if there exists a grading operator $D$ such that $\Phi_D$ is an $O(f(r))$-SBE?
\item\label{sbeni_c2} is it true that 
$G$ is $O(f(r))$-SBE to $G_\infty$ if and only if there exists a grading operator $D$ such that $\Phi_D$ is an $O(f(r))$-SBE?
\end{enumerate}
\end{que}

Question \ref{sbeni_c}(\ref{sbeni_c1}) might have a negative answer because of some unexpected simplification in the computation involving the Baker-Hausdorff formula (computer assistance might help finding such a putative counterexample); this is why I have not conjectured a positive answer to Conjecture \ref{q_sbenil}. In case of such a negative answer, Question \ref{sbeni_c}(\ref{sbeni_c2}) sounds like a replacement, saying that the ``best" SBE between one simply connected Lie group and its Lie algebra should be found among the $\Phi_D$.

\begin{ex}\label{e_dehn}
Consider the 6-dimensional nilpotent real Lie algebra $\g$ defined as central product of a 4-dimensional and a 3-dimensional filiform Lie algebra. It can be described by its nonzero brackets: $[e_1,e_2]=e_5$, $[e_1,e_5]=[e_3,e_4]=e_6$. (This is $\g_{6,2}$ in \cite{Mag} and $L_{6,10}$ in \cite{graaf}.) This Lie algebra is not Carnot: the associated Carnot Lie algebra $\g_\infty$ has the same brackets except that $[e_3,e_4]=0$, and has 3-dimensional center while $\g$ has 1-dimensional center. Also, they have distinct Betti numbers: $b_2(\g')=7$ while $b_2(\g)=6$. This implies that they are not quasi-isometric, by Shalom's theorem \cite{sh04}. More precisely, under the Carnot grading, $\g'$ has 1-dimensional 2-homology in degree 4, which corresponds to a central extension of nilpotency length equal to 4, while $\g$ has no such central extension (every element of $\Lambda^2\g$ of the form $x\wedge z$, $(x,z)\in\g\times\g^3$ is a boundary). 

Let $G,G'$ be the corresponding simply connected Lie groups; we have $e_G=2/3$. For which $\alpha$ are $G$ and $G'$ $O(r^\alpha)$-SBE? The reason we insist on this example is that we expect that it might be possible to obtain nontrivial lower bounds on such $\alpha$ using the Dehn function, see Question \ref{dehn}. 
\end{ex}

\begin{que}\label{dehn}
Continue with the notation of Example \ref{e_dehn}. Is it true that $G$ has Dehn function in $o(r^4)$?
\end{que}

The existence of a 4-nilpotent central extension of $\g'$ implies, by standard arguments, that $G'$ has Dehn function $\simeq r^4$. Thus a positive answer to Question \ref{dehn} would provide the first example of a nilpotent group with Dehn function not equivalent to that of its associated Carnot group. Furthermore, a positive answer with a reasonably explicit upper bound and description of the homotopy might imply a positive lower bound on the set of $\alpha$ such that $G$ and $G'$ are SBE, which would be the first known results improving the bare fact that they are not quasi-isometric.

\section{Large-scale contractions and similarities}\label{LSCS}

\subsection{Large-scale contractions in the quasi-isometric setting}

In metric spaces, a $c$-Lipschitz self-map for $c<1$ has very simple dynamical properties. We extend this here to the large-scale setting, so as to obtain a reasonable definition of large-scale contractable metric space.

\begin{defn}\label{liplip}
Let $X,Y$ be metric spaces and $f:X\to Y$. We say that $f$ is $(c,C)$-Lipschitz' if $d(f(x),f(x'))\le\max(cd(x,x'),C)$ for all $x,x'\in X$. We say that $f$ is $c$-LS-Lipschitz' if it is $(c,C)$-Lipschitz' for some $C$.
\end{defn}

\begin{rem}\label{lipstar}
There is a more usual closely related variant: say that $f$ is $(c,C)$-Lipschitz if $d(f(x),f(x'))\le cd(x,x')+C$ for all $x,x'\in X$. Then if $f$ is $(c,C)$-Lipschitz', then it is $(c,C)$-Lipschitz. Conversely, if $f$ is $(c,C)$-Lipschitz, then it is $(c',C')$-Lipschitz' for every $c'>c$ and $C'\ge C/(1-c/c')$.

In particular, if $f$ is $c$-LS-Lipschitz' then it is $c$-LS-Lipschitz, and if $f$ is $c$-LS-Lipschitz, then it is $c'$-LS-Lipschitz' for all $c'>c$. As a consequence, if $g:X\to Y$ is at bounded distance of a $c$-LS-Lipschitz' map, then it is a $c'$-LS-Lipschitz' map for all $c'>c$ (indeed, $g$ is clearly $c$-LS-Lipschitz).
\end{rem}

\begin{lem}\label{composels}
Let $f:X\to Y$ be a $(c,C)$-Lipschitz' map and let $g:Y\to Z$ be a $(c',C')$-Lipschitz' map. Then $g\circ f$ is $(cc',\max(c'C,C'))$-Lipschitz'. In particular, 
\begin{itemize}
\item if $f$ is $c$-LS-Lipschitz' and $g$ is $c'$-LS-Lipschitz' then $g\circ f$ is $cc'$-LS-Lipschitz';
\item $f$ is $c$-LS-Lipschitz', $n\ge 0$ implies that $f^n$ is $c^n$-LS-Lipschitz';
\item $f$ is $(c,C)$-Lipschitz' with $c\le 1$, $n\ge 0$ implies that $f^n$ is $(c^n,C)$-Lipschitz'.
\end{itemize}
\end{lem}
\begin{proof}We have the inequality
\[d(g\circ f(x),g\circ f(x'))\le \max(c'd(f(x),f(x')),C')\]\[\le \max(c'\max(cd(x,x'),C),C')=\max(cc'd(x,x'),\max(c'C,C')).\]
All three consequences are immediate.
\end{proof}

\begin{defn}
We say that $X$ is LS-contractable if it admits a LS-contraction, i.e., a self-quasi-isometry that is $c$-LS-Lipschitz' for some $c<1$.
\end{defn}

Taking powers and using Lemma \ref{composels}, we see that this implies that $X$ admits self-quasi-isometries that are $(c,C)$-LS-Lipschitz' for $c$ arbitrary close to 0 and $C$ uniform. 

\begin{prop}\label{lsc_qi}
To be LS-contractable is a QI-invariant.
\end{prop}
\begin{proof}
Let $X$ be a LS-contractable metric space. Let $Y$ be  metric space, and let $u:X\to Y$ and $v:Y\to X$ be quasi-isometries, say both $\lambda$-LS-Lipschitz' for some $\lambda>0$. Consider a $c$-LS-Lipschitz' self-quasi-isometry $f$ of $X$ for $c<\lambda^{-2}$. Then $u\circ f\circ v$ is a $c\lambda^2$-LS-Lipschitz' self-quasi-isometry of $Y$, and $c\lambda^2<1$. 
\end{proof}

\begin{prop}\label{lsc_po}
Let $X$ be a nonempty connected graph of bounded degree. Suppose that $X$ is LS-contractable. Then $X$ has polynomial growth.
\end{prop}
\begin{proof}Let $f$ be a $(c,C)$-Lipschitz' self-quasi-isometry of $X$ (we can suppose that $C>0$); suppose that $X$ has degree $\le\delta$. It follows that the fibers of $f$ have cardinal at most $k$, for some integer $k$.

For $n\ge 0$, define $Z_n=\{x\in X:d(x,f(x))\le c^{-n}C\}$. We see that $Z_0\subset Z_1\subset\dots$, and, for every $n\ge 1$, $f(Z_n)\subset Z_{n-1}$, i.e., $Z_n\subset f^{-1}(Z_{n-1})$. This implies that $\#(Z_n)\le k^n\#(Z_0)$ for all $n$.

Suppose $z,z'\in Z_0$ with $d(z,z')\ge c^{-1}C$. Then 
\[d(f(z),f(z'))\le cd(z,z')\le cd(z,f(z))+cd(f(z),f(z'))+cd(z',f(z'))\]
\[\le 2cC+cd(f(z),f(z')),\]
whence 
\[d(f(z),f(z'))\le \frac{2cC}{1-c};\]
there exists $c''$ such that for all $x,x'\in X$, $d(f(x),f(x'))\le \frac{2cC}{1-c}$ implies $d(x,x')\le c''$. Hence the diameter of $Z$ is $\le\max(c^{-1}C,c'')$; in particular $Z_0$ is bounded, hence finite.

Fix $x_0\in Z_0$. If $d(x,x_0)\ge c^{-1}C$, then
\[d(x,f(x))\le d(x,x_0)+d(x_0,f(x_0))+d(f(x),f(x_0))\]
\[\le (1+2c)d(x,x_0);\]
hence, denoting by $B(r)$ the closed $r$-ball around $x_0$, we have
\[B((1+2c)c^{-n}C)\smallsetminus B(c^{-1}C)\subset Z_n;\]
hence 
\[\#B((1+2c)c^{-n}C)\le \#B(c^{-1}C)+k^n\#(Z_0)\]
for all $n\ge 0$, and hence $X$ has polynomial growth (of degree $\le\log_{c^{-1}}(k)$).
\end{proof}

\begin{rem}
The proof works more generally in any uniformly discrete metric space, which means a metric space such that balls of radius $r$ have cardinal $\le u(r)$ for some function $u$ (see \cite[\S 3.D]{CH}).
\end{rem}

\begin{cor}
If a compactly generated, locally compact group $G$ is LS-contractable, then it has polynomial volume growth.
\end{cor}
\begin{proof}
Indeed $G$ is quasi-isometric to some connected, bounded degree graph (see \cite[\S 3.B]{CH} if necessary), which has to be LS-contractable, hence of polynomial growth.
\end{proof}

Therefore the characterization of compactly generated locally compact groups that are LS-contractable reduces to the case of simply connected nilpotent Lie groups. If such a group admits a contracting automorphism, then it is coarsely contractable. 

\begin{que}\label{lsc_q}
Conversely, if a simply connected nilpotent Lie group is LS-contractable, does it admit a contracting automorphism?
\end{que}

We rather expect a positive answer, although we do not know a single example of a simply connected nilpotent Lie group that is not LS-contractable (while many simply connected nilpotent Lie groups in dimension 7 and higher, have no contracting automorphisms). A simply connected nilpotent Lie group has a contracting automorphism if and only if its algebra has one, which is equivalent to having a grading in positive integers (not necessarily Carnot).

\subsection{Sublinearly contractable}\label{s_sc}

We can define, more generally, a $O(u)$-contraction to be a $(c,O(u))$-Lipschitz' and $O(u)$-SBE map for some $c<1$, and an $o(u)$-contraction to be an $O(v)$-contraction for some $v=o(u)$.

We say that a space is $O(u)$-contractable, respectively $o(u)$-contractable, if it admits an $O(u)$-contraction, resp.\ $o(u)$-contraction. Denoting as we do here the variable by $r$, we write ``sublinearly contractable" for $o(r)$-contractable. Arguing as in the large-scale case (that is, $O(1)$-contractibility), this is a $O(u)$-SBE (resp.\ $o(u)$-SBE) invariant.

Beware that, in constast to Proposition \ref{lsc_po}, a sublinearly contractable bounded degree graph need not be of polynomial growth: indeed, the trees of superpolynomial growth of Example \ref{counter_tre}(\ref{polgr_ni}) are SBE to a geodesic ray, and hence sublinearly contractable.

In particular, every compactly generated locally compact group with polynomial group being SBE to a Carnot group, it is sublinearly contractable.
More precisely, every simply connected nilpotent Lie group $G$ with $e=e_G$ is $O(r^e)$-contractable. But this is not always optimal. Indeed, it can happen that $G$ is contractable (i.e., admits a positive grading) and hence is $O(1)$-contractable (although not always being quasi-isometric to a Carnot group). This applies to the 5-dimensional 3-step nilpotent non-Carnot case.

It can also happen that $G$ is not contractable, but is known to be $O(r^\alpha)$-contractable for some $\alpha<e_G$, as is given by the following proposition.

\begin{prop}\label{708}
The following characteristically nilpotent Lie algebra $\g$, denoted $\g_{7,0,8}$ in \cite{Mag}, with basis $(e_i)_{1\le i\le 7}$ and nonzero brackets (writing $ij\vert k$ for $[e_i,e_j]=e_k$)
\[12\vert 4,\;14\vert 5,\;15\vert 6,\;26\vert 7,\;54\vert 7,\quad 13\vert 7,\;23\vert 6,\;24\vert 6,\]
has $e_\g=3/4$, but the corresponding simply connected (7-dimensional real) Lie group is $O(r^{2/5})$-contractable. More precisely, it is $O(r^{2/5})$-SBE to a contractable Lie group (i.e., whose Lie algebra admits a positive grading).
\end{prop}
\begin{proof}
The nilpotency length of $\g$ is 5, and the quotient $\g/\g^5$ is isomorphic to $\g_{6,12}$, which has $e=3/4$ (see \S\ref{misce}). In particular $e_\g\ge 3/4$. Let $D$ be the grading operator corresponding to the grading for which $e_1,e_2,e_3$ have degree 1 and $e_i$ has degree $i-2$ for $i=4,5,6,7$. Then among the brackets above, only the last three fail to respect the degree. It readily follows that quadruple brackets respect the degree. In particular, $D\in\mathcal{D}_\wp^6(\g)$ for all $\wp$ such that $|\wp|=4$. Since for $2\le i<j\le 5$ and $i\le 3$ we have $i/j\le 3/4$, we deduce $e_\g\le 3/4$.

Again using the nomenclature in \cite{Mag}, the Lie algebra $\mathfrak{h}=\g_{7,1,21}$ is defined exactly with the same brackets, except that $[e_1,e_3]=0$. This Lie algebra admits a positive Lie algebra grading (for which $e_1$ has degree 1, $e_2$ has degree 2, $e_3,e_4$ have degree 3, $e_5$ has degree 4, $e_6$ has degree 5, and $e_7$ has degree 6). So, identifying the Lie algebras with their exponentials, it is enough to show that the identity map between the two group laws is an $O(r^{2/5})$-SBE.

Denote by $B'$ the bracket in $\g$ and $B$ the bracket in $\mathfrak{h}$. Then the alternating bilinear form $B'-B$ maps $e_1\wedge e_3$ to $e_7$ and other basis elements to 0. As for brackets, write iterated brackets as $B(x,B(y,z))=B(x,y,z)$, etc. Since its image of $B'-B$ is central (for both brackets), any iterated bracket $B'(x_1,\dots,x_n)$ is equal to $B'(x_1,B(x_2,\dots,x_n))$. Since moreover $B'-B$ vanishes on $\g\times\g^2$, we deduce that for $n\ge 3$, $B'(x_1,\dots,x_n)=B(x_1,\dots,x_n)$ for all $n$. Therefore, denoting by $\ast$ and $\ast'$ the corresponding laws given by the Baker-Hausdorff formula, only terms of degree 2 remain when computing:
\[(x\ast' y)-(x\ast y)=\frac12(B'(x,y)-B(x,y)),\quad\forall x,y\in\g.\]
Decomposing $x$ and $y$ along the basis $(e_i)$, we get 
\[(x\ast' y)-(x\ast y)=\frac12(x_1y_3-x_3y_1)e_7.\]
Since $\lfloor x\rfloor$ is equal to $\max(|x_1|,|x_2|,|x_3|,|x_4|^{1/2},|x_5|^{1/3},|x_6|^{1/4},|x_7|^{1/5})$ (after some choice of submultiplicative norm),
this implies an equality of the form $\lfloor(x\ast' y)-(x\ast y)\rfloor\le C\max(\lfloor x\rfloor,\lfloor y\rfloor)^{2/5}$. 
We can then end the proof exactly as in the proof of Theorem \ref{resbe}.
\end{proof}

\subsection{Large-scale similarities}\label{s_lss}

A similarity of a metric space is a self-map that pulls back the distance to a multiple of itself. We define here the large-scale analogue of this notion. Write $\log^+(x)=\max(0,\log x)$.

\begin{defn}
Let $f$ be a $(c,C)$-Lipschitz' self-map of $X$ with $c<1$ (in the sense of Definition \ref{liplip}). For $t\ge C$, define $\Lambda_t=\Lambda_t^{f}$ by 
\[\Lambda_t(x,x')=\min\{n\in\N:d(f^n(x),f^n(x'))\le t\}.\] Note that $\Lambda_t-\Lambda_C$ is bounded on $X\times X$ (if $t\le c^{-n}C$ for some $n\ge 0$, then it is bounded by $n$).

Given $b\in\mathopen]0,1\mathclose[$, we say that $f$ is a $b$-LS-similarity if $\Lambda_t-\log_{b^{-1}}^+d$ is bounded on $X\times X$ for some/any $t\ge C$.
\end{defn}

\begin{ex}~
\begin{enumerate}
\item Let $\R$ be endowed with the usual metric and $f(x)=bx$, with $0<b<1$. Then for any $t>0$ and $x,x'\in\R$ we have $\Lambda_t(x,x')=\lceil\log_{b^{-1}}^+(|x-x'|/t)\rceil$. In particular $\Lambda_t^f(x,x')=\log_{b^{-1}}^+(|x-x'|)+O(1)$ (in the sense that the difference is bounded is, for given $t>0$, a bounded function of $(x,x')$).
\item Let $\R^2$ be endowed with the sup norm, and $g(x,y)=(x/2,y/3)$. Then \[\Lambda^g_t((x,y),(x',y'))=\lceil\max(\log_2^+(|x-x'|/t),\log_3^+(|y-y'|/t))\rceil.\]
In particular, 
\[\Lambda_t^g((x,y),(x',y'))=\max(\log_2^+(|x-x'|),\log_3^+(|y-y'|))+O(1).\]
\item Let $\R^2$ be endowed with any norm, fix $0<b<1$ and $h(x,y)=b^{-1}(x+y,y)$. Then for any $t>0$, \[\Lambda_t^h((x,0),(x',0))=\log_{b^{-1}}^+(|x-x'|)+O(1),\] and \[\Lambda_t^h((0,y),(0,y'))=\log_{b^{-1}}^+(|y-y'|)+\log_{b^{-1}}^+(\log_{b^{-1}}^+(|y-y'|))+O(1).\]
\end{enumerate}
\end{ex}

\begin{rem}
If $X$ is unbounded, then \[\log^+_{b^{-1}}d-\log^+_{{b'}^{-1}}d=\left(\frac{1}{\log(b^{-1})}-\frac1{\log(b'^{-1})}\right)\log^+d\] is unbounded as soon as $b\neq b'$, and it follows that $f$ is a $b$-LS-similarity for at most one value of $b$. (If $X$ is bounded, then every map $f:X\to X$ is obviously a $b$-LS-similarity for every $b$.)
\end{rem}

\begin{lem}
Let $f$ be a $c$-LS-Lipschitz' self-map of $X$ with $c<1$. Let $w:X\to X$ have bounded distance to the identity map of $X$. Then for large $t$ we have $|\Lambda_t^f-\Lambda_t^{f\circ w}|$ bounded. In particular, if $f$ is a $b$-LS-similarity for $b<1$ then $f\circ w$ is a $b$-LS-similarity.
\end{lem}
\begin{proof}
We already know from Remark \ref{lipstar} that $f\circ w$, being at bounded distance to $f$, is $c'$-LS-Lipschitz' for any $c'\in\mathopen]c,1\mathclose[$. Let $w$ have displacement $\le k$; we can suppose $k\ge c^{-1}C$. Then \[d(f^n(x),(f\circ w)^n(x))\le (c+\dots+c^n)k\] for all $n$.

Proof by induction (we write $fw$ for $f\circ w$): trivially true for $n=0$, and:
\[d(f^{n+1}(x),(fw)^{n+1}(x))\le \max\Big(C,cd\big(f^n(x),w(fw)^n(x)\big)\Big)\]
\[\le \max\Big(C,cd\big(f^n(x),(fw)^n(x)\big)+cd\big((fw)^n(x),w(fw)^n(x)\big)\Big)\]
\[\le \max\big(C,c(c+\dots+c^n)k+ck)\big)=(c+\dots +c^{n+1})k\]

Thus if $c_\infty=\sum_{n\ge 1}c^n=c/(1-c)$, then 
\[d(f^n(x),(fw)^n(x))\le c_\infty k,\quad \forall n\ge 0\] 

Thus $\Lambda_{t+2c_\infty k}^{fw}\le \Lambda_t^f$ and $\Lambda_{t+2c_\infty k}^{f}\le \Lambda_t^{fw}$. Hence $\Lambda_t^f-\Lambda_t^{fw}$ is bounded for large $t$ (say $t\ge C+2c_\infty k$).
\end{proof}

\begin{warn}
In the Euclidean plane $\R^2$, let $f$ be the diagonal matrix $(1/2,1/3)$ and $u$ be the flip of coordinates. Then $\Lambda^{f}-\Lambda^{ufu}$ is unbounded. Thus, conjugating by a bijective quasi-isometry does not preserve the class of $\Lambda^f$ modulo bounded functions. 
\end{warn}

\begin{prop}
Let $f$ be a $b$-LS-similarity of $X$, and $u:X\to Y$, $v:Y\to X$ be inverse quasi-isometries, and define $g=ufv:Y\to Y$. Then $g$ is a $b$-LS-similarity of $Y$.
\end{prop}
\begin{proof}
For the moment, only assume that $f$ is $(c,C)$-Lipschitz'. Again by Remark \ref{lipstar}, $g$ is $c'$-LS-Lipschitz' for any $c'\in\mathopen]c,1\mathclose[$.

Writing $w=vu$. Letting $k$ be such that $w$ has displacement $\le k$, and let $m,m'$ be such that $d(ux,ux')\le\max(m'd(x,x'),m')$ for all $y,y'\in Y$. Then we have, on the one hand
\[d((ufv)^n(y),(ufv)^n(y'))=d\big(u(fw)^{n-1}(fvy),u(fw)^{n-1}(fvy')\big)\]
\[\le \max\Big(m',md\big((fw)^{n-1}(fvy),(fw)^{n-1}(fvy')\big)\Big);\]
since $d((fw)^{n-1},f^{n-1})\le c_\infty k$ by the previous proof, we have
\[d((ufv)^n(y),(ufv)^n(y'))\le \max\Big(m',2mc_\infty k+md\big(f^nvy,f^nvy')\big)\Big).\]
Now specify this to $n=n_0=\Lambda_C^f(vy,vy')$. Then we obtain 
\[d((ufv)^{n_0}(y),(ufv)^{n_0}(y'))\le M=\max\Big(m',2mc_\infty k+mC\Big).\]
Hence $\Lambda_M^{ufv}(y,y')\le n=\Lambda_C^f(vy,vy')$. 

On the other hand, assuming $d(fv(y),fv(y'))\le\max(m'_1,m_1d(y,y'))$ for all $y,y'$, and letting again $n$ be arbitrary,

\[d((fvu)^n(x),(fvu)^n(x'))=d\big(fv(ufv)^{n-1}u(x),fv(ufv)^{n-1}u(x')\big)\]
\[\le \max\Big(m'_1,m_1d\big((ufv)^{n-1}u(x),(ufv)^{n-1}u(x')\big)\Big) \]

Fix $t$ (large enough), and set $n=\Lambda^{ufv}_t(u(x),u(x'))+1$. Then 
\[d((fvu)^n(x),(fvu)^n(x'))\le M'=\max\Big(m'_1,m_1t\Big).\]
Hence $\Lambda^{fvu}_{M'}(x,x')\le n=\Lambda^{ufv}_t(u(x),u(x'))+1$

If $t'\ge M'$ is large enough so that $\Lambda_{t'}^{fvu}-\Lambda_{t'}^f$ is bounded, then we deduce that $\Lambda_{t'}^{f}-\Lambda_{t}^{ufv}\circ (u\times u)$ is upper bounded, as well as $\Lambda_M^{ufv}-\Lambda_C^f\circ (v\times v)$.

If we assume that $f$ is a $b$-LS-similarity, then $\Lambda^f_C-\Lambda^f_{C}\circ (v\times v)$ is bounded. Hence $\Lambda_M^{ufv}-\Lambda_C^f$ is upper bounded. Also 
$\Lambda_{t'}^{f}-\Lambda_{t'}^{f}\circ (u\times u)$ is bounded, whence $(\Lambda_{t'}^{f}-\Lambda_{t}^{ufv})\circ (u\times u)$ is upper bounded, and hence $(\Lambda_{t'}^{f}-\Lambda_{t}^{ufv})$ is upper bounded as well. Hence $(\Lambda_{t'}^{f}-\Lambda_{t}^{ufv})$ is bounded. So $ufv$ is a $b$-LS-similarity.
\end{proof}

Note that if $f$ is an LS-similarity, then it is coarsely proper. In particular, if $f$ is an essentially surjective LS-similarity and $X$ is large-scale geodesic, then $f$ is a self-quasi-isometry.

\begin{defn}
We say that $X$ is large-scale homothetic (LSH) if it admits, for some $b\in\mathopen]0,1\mathclose[$, a $b$-LS-similarity that is a self-quasi-isometry.
\end{defn}

\begin{cor}\label{lsh_qi}
Being large-scale homothetic is a quasi-isometry invariant.\qed
\end{cor}

Note that this implies that $X$ is LS-contractable. An easy instance of a space that is LS-contractable but not LSH is the set of $2^{2^n}$ when $n$ ranges over positive integers.

If a simply connected nilpotent Lie group is Carnot, then it is LSH, since for a good choice of metric equivalent to the word length (Carnot-Caratheodory), it admits non-isometric self-similarities.

\begin{que}\label{q_lshcarnot}
Conversely, if a simply connected nilpotent Lie group is LSH, is it Carnot?
\end{que}

\end{document}